\documentclass[11pt,leqno]{amsart}
\usepackage{amsthm,amsfonts,amssymb,amsmath,oldgerm}
\numberwithin{equation}{section}
\usepackage{fullpage}
\usepackage{color}

\def\eps{\varepsilon }

\newcommand\R{\mathbb R}
\newcommand\C{\mathbb C}

\def\eps{\varepsilon}

\newcommand\kernel{\hbox{\rm Ker}}

\newcommand\br{\begin{remark}}
\newcommand\er{\end{remark}}
\newcommand\bp{\begin{pmatrix}}
\newcommand\ep{\end{pmatrix}}
\newcommand{\be}{\begin{equation}}
\newcommand{\ee}{\end{equation}}
\newcommand\ba{\begin{equation}\begin{aligned}}
\newcommand\ea{\end{aligned}\end{equation}}

\newcommand{\bap}{\begin{app}}
\newcommand{\eap}{\end{app}}
\newcommand{\begs}{\begin{exams}}
\newcommand{\eegs}{\end{exams}}
\newcommand{\beg}{\begin{example}}
\newcommand{\eeg}{\end{exaplem}}
\newcommand{\bpr}{\begin{proposition}}
\newcommand{\epr}{\end{proposition}}
\newcommand{\bt}{\begin{theorem}}
\newcommand{\et}{\end{theorem}}
\newcommand{\bc}{\begin{corollary}}
\newcommand{\ec}{\end{corollary}}
\newcommand{\bl}{\begin{lemma}}
\newcommand{\el}{\end{lemma}}
\newcommand{\bd}{\begin{definition}}
\newcommand{\ed}{\end{definition}}
\newcommand{\brs}{\begin{remarks}}
\newcommand{\ers}{\end{remarks}}

\newcommand{\RR}{{\mathbb R}}

\newcommand{\ZZ}{{\mathbb Z}}

\newcommand{\CC}{{\mathbb C}}

\newcommand{\pa}{{\partial}}
\newcommand{\const}{\text{\rm constant}}
\newcommand{\Id}{{\rm Id }}

\newcommand{\Span}{{\rm Span }}

\newcommand{\sgn}{\text{\rm sgn}}
\newtheorem{theorem}{Theorem}[section]
\newtheorem{proposition}[theorem]{Proposition}
\newtheorem{corollary}[theorem]{Corollary}
\newtheorem{lemma}[theorem]{Lemma}

\theoremstyle{remark}
\newtheorem{remark}[theorem]{Remark}
\theoremstyle{definition}
\newtheorem{definition}[theorem]{Definition}

\newtheorem{example}[theorem]{Example}

\newcommand{\RM}{\mathbb{R}}

\newcommand{\HM}{\,\mbox{\bf H}}

\newcommand{\dom}{\text{\rm{dom}}}

\newcommand{\beq}{\begin{equation}}
\newcommand{\eeq}{\end{equation}}

\title{
Quasi-gradient systems, modulational dichotomies,\\
and stability of spatially periodic patterns
}

\author{ Alin Pogan}
\address{ Indiana University, Bloomington, IN 47405}
\email{apogan@indiana.edu}
\thanks{Research of A.P. was partially supported under
NSF grant no.  DMS-0806614}
\author{Arnd Scheel}
\address{University of Minnesota,
206 Church St. SE, Minneapolis, MN 55455}
\email{scheel@ima.umn.edu}
\thanks{Research of A.S. was partially supported under
NSF grant no.  DMS-0806614}
\author{Kevin Zumbrun}
\address{Indiana University, Bloomington, IN 47405}
\email{kzumbrun@indiana.edu}
\thanks{Research of K.Z. was partially supported
under NSF grant no. DMS-0300487}
\begin{document}

\begin{abstract}
Extending the approach of
Grillakis--Shatah--Strauss, Bronski--Johnson--Kapitula, and others
for Hamiltonian systems, we explore relations between the
constrained variational problem $\min_{X:C(X)=c_0} \mathcal{E}(X)$,
$c_0\in \RM^r$, and stability of solutions of
a class of degenerate ``quasi-gradient'' systems
$dX/dt=-M(X)\nabla \mathcal{E}(X)$
admitting constraints,
including Cahn-Hilliard equations, one- and multi-dimensional viscoelasticity,
and coupled conservation law-reaction diffusion systems
arising in chemotaxis and related settings.
Using the relation between variational stability and the signature of
$\partial c/\partial \omega \in \R^{r\times r}$,
where $c(\omega)=C(X^*_\omega)\in \RM^r$ denote
the values of the imposed constraints
and $\omega\in \RM^r$ the associated Lagrange multipliers at
a critical point $X^*_\omega$,
we obtain as in the Hamiltonian case a
general criterion for co-periodic stability of periodic waves,
illuminating and extending a number of previous results obtained by direct
Evans function techniques.
More interestingly, comparing the form of the Jacobian
arising in the co-periodic theory to Jacobians arising in the
formal Whitham equations associated with modulation,
we recover and substantially generalize
a previously mysterious ``modulational dichotomy'' observed in special cases
by Oh--Zumbrun and Howard, showing that co-periodic and sideband
stability are incompatible. In particular, we both illuminate and
extend to general viscosity/strain-gradient effects and
multidimensional deformations the result of Oh-Zumbrun of universal
modulational instability of periodic solutions of the equations of
viscoelasticity with strain-gradient effects, considered as
functions on the whole line. Likewise, we generalize to
multi-dimensions corresponding results of Howard on periodic
solutions of Cahn--Hilliard equations.
\end{abstract}
\date{\today}
\maketitle

\section{Introduction }\label{s:intronew}

In this paper, we investigate from a general point of view stability
of standing spatially periodic solutions of a broad class of equations with
gradient or ``quasi-gradient'' structure including Cahn--Hilliard,
viscoelasticity, and coupled conservative--reaction diffusion equations.
Our main goals are to (i) understand modulational stability of periodic
viscoelastic solutions with physically relevant
viscosity/strain-gradient coefficients and multi-dimensional deformations,
about which up to now very little has been known, and
(ii) understand through gradient or related structure the
commonality of results that have emerged through direct Evans-function
techniques in this and other types of systems.

In particular, we seek to better understand an apparently special
modulational dichotomy that was established for a special model of
one-dimensional viscoelasticity by
Oh--Zumbrun in \cite{OZ1}, but later shown by Howard \cite{Ho1}
to hold also for the Cahn--Hilliard equation, asserting that,
under rather general conditions, {\it co-periodic
and sideband stability are mutually incompatible:}
in particular, periodic solutions are always modulationally unstable.
Both of these results were obtained by direct Evans-function computations,
leaving the underlying mechanism unclear.
Moreover, in the viscoelastic case, the technical condition of monotonicity
of period with respect to amplitude left open the possibility that there might
be special parameter regimes in which periodic waves could be stable.
Further, the method of proof does not appear to generalize to multi-dimensional
planar periodic waves, as studied for example in \cite{Y,BYZ}.

Thus, these prior results left intriguing questions regarding the possibility
of stable periodic viscoelastic solutions.
We note that, even though second-order reaction-diffusion equations
cannot, by Sturm--Liouville considerations, support stable periodic waves,
nor, for the same reasons, can conservative scalar Cahn--Hilliard equations,
the Swift-Hohenberg equations give a well-known example of
a higher-order reaction-diffusion equation that does
support stable waves \cite{Mi,Sc},
and generalized Kuramoto-Sivashinsky equations
are another well-known example showing that conservation laws
can likewise support stable periodic waves; see \cite{JNRZ1,BJNRZ,FST}.
Thus, the conservative structure of the equations of viscoelasticity
is not a priori an obstruction to stability.

A clue noted in \cite{OZ1,Z3} regarding possible common origins of
the modulational dichotomies for Cahn--Hilliard and viscoelasticity equations
are that both serve as models for phase-transitional elasticity, and
both share a nonincreasing Cahn--Hilliard/van der Waals energy
functional $\mathcal{E}$.
Indeed, the Cahn--Hilliard equation is well-known to be the gradient
system with respect to $H^{-1}$ of the functional $\mathcal{E}$,
that is,
\be\label{struct}
(d/dt)X(t)=-M\nabla \mathcal{E}(X(t)),
\ee
where $M= (i\partial_x)(i\partial_x)^*=-\partial_x^2$ is symmetric
positive semidefinite.
It has been conjectured \cite{OZ1,Z3}
that the similar results obtained for viscoelasticity are due to
this ``hidden variational structure.''

Here, we show that this is true, in an extremely strong sense.
Namely, we show that the equations of viscoelasticity with strain-gradient
effects may be put into the same form \eqref{struct}, but with $M$ only
{\it positive semidefinite} (i.e., $M+M^*\ge 0$) and not symmetric-
a structure that might be called ``generalized gradient'' or
``quasi-gradient'' form.
More, we show that, under an additional property \eqref{e1}
of viscoelasticity, satisfied automatically in the gradient case
$M=M^*$,
all of the results that have so far been shown for this
system, including the results on modulational dichotomy,
may be deduced entirely
from this quasi-variational structure,
constrained variational stability considerations like those used
in \cite{GSS,BJK} and elsewhere to study stability of Hamiltonian waves,
and a novel connection observed between the Jacobian 
and Lagrangian multipliers arising in constrained variational theory
and the Whitham modulation equations formally governing sideband behavior.
The latter we regard as the main new contribution of this paper.

Indeed, we are able to put each of the mentioned applications into
this framework and obtain a suite of similar results by these common
methods.
As in any investigation overlapping with classical calculus of variations
and Sturm--Liouville theory, many of our results in the simpler
cases considered here can be obtained by other means, and a great many
have (no doubt more than the ones we have been able to cite).
Still more of our results can be obtained alternatively in a slightly weakened
form by direct Evans function computations, which give a sort of
mod two version of Sturm--Liouville applying in more general circumstances;
see for example, the discussion of Sections \ref{s:SL} and \ref{ozrmk}.

However, a significant number of our results appear also to be new,
in particular (surprisingly) modulational instability of
multiply periodic Cahn--Hilliard solutions (though see Remark \ref{arndrmk}
on the two-dimensional case), Theorem \ref{V-E-W-dich};
modulational instability of singly and multiply periodic coupled
conservation--reaction diffusion (in particular, Keller--Segel) solutions,
Theorem \ref{cRD-dich};
universal modulational instability, without the technical condition of
monotonicity of period vs. amplitude, of one-dimensional periodic
viscoelasticity waves, Theorem 5.1; and modulational instability of the
branch of planar multi-dimensional deformation solutions of viscoelasticity
connected by transversal homotopy to the decoupled, one-dimensional case,
Corollary \ref{contcor}.
By related techniques,
we recover also versions of the results of \cite{GZ,Z3,PS2}
on instability of homoclinic ``spike'' type solutions on the whole line;
see Section \ref{s:spikes}.

Also, though many of these ideas
appear elsewhere in  disparate sources,
we feel that there is some value in putting a broad variety of
problems into a common framework, continuing and
extending the unification initiated in \cite{PS1,PS2}.
In particular, we both formalize the heuristic
relation pointed out in \cite{PS2} between
coupled conservative--reaction diffusion
and Cahn--Hilliard equations and extend this to the quasi-gradient
viscoelastic case.
We discuss also the natural connection to the Hamiltonian case
of systems $(d/dt)X(t)=J\nabla \mathcal{E}(X(t))$ with $J=-J^*$ skew
symmetric.
Here, interestingly,
the appearance of conservation laws for generators
$-MA$ with $M$ noninvertible is paralleled by the appearance for
$J$ noninvertible of conservation
laws associated with noninvertibility of $J$ rather than Noetherian
symmetries; see \eqref{second_class}.
Finally, in Section \ref{s:revisit}, we give a unified explanation
of the previously-observed modulational dichotomies from the variational
point of view, showing that the mechanisms lying behind these are
unconstrained instability of $A$, together with
a relation between Morse
indices of the Floquet operators associated with $MA$ and $A$,
and the aforementioned property \eqref{e1} that the
center subspace of $MA$ consist of $\ker MA$.

\subsection{Discussion and open problems}
As concerns our primary interest of classifying stability of
periodic viscoelastic patterns, we have resolved some of the main
problems pointed out in \cite{OZ1,Y}, but also pointed the way
to some interesting further directions for investigation.
In particular, it is an extremely interesting open
question whether there could be a bifurcation to a stable branch
of planar multideformational patterns as solutions are continued
from the decoupled, effectively one-dimensional deformation, case.
Another interesting open question is whether
there are there other periodic viscoelastic
solutions that are `genuinely multi-dimensional''
in the sense that they are not continuable from the decoupled case.
These and related problems for homoclinic solutions are being investigated
in \cite{BYZ} by numerical Evans function techniques.

We note that an efficient alternative Evans method is proposed here
for self-adjoint problems, based on the Sturm-Liouville duality
between spatial and frequency variables $x$ and $\lambda$;
see the last paragraph of Section \ref{s:SL}.
This also would be interesting to explore.
Finally, it would be interesting to look for stable multiply
periodic viscoelasticity solutions, perhaps corresponding to waves
seen in buckling of viscoelastic materials \cite{HLFY,BCS});
see Section \ref{s:multivisc}.

\medskip
{\bf Note.}  We have been informed after the completion of this
work that our analysis in Section \ref{s:CH}
overlaps with unpublished results of Bronski and Johnson on
nonlocal Cahn--Hilliard equations.

\section{Preliminaries: variational stability and Sylvester's law
of inertia}\label{s:prem-res}

We begin by recalling a fundamental relation (\eqref{sigrel} below)
underlying the Hamiltonian stability results of \cite{GSS}, \cite{BJK}, etc.,
between stability of constrained critical points and the Jacobian
$\partial c/\partial \omega$ of the map
from Lagrange multipliers to constraints.
This principle seems to go far back, at least implicitly, in the physics
(and particularly the thermodynamics) literature \cite{C,LL}.
For a general treatment from the Calculus of Variations point of
view, see for example \cite{M} and references therein.\footnote{
Curiously, the Jacobian $(\partial c/\partial\omega)( \omega^*)$
emphasized in \cite{GSS} and later works is never explicitly
mentioned in \cite{M}, even though the physical importance of the
Lagrange multiplier is stressed. Rather, it appears in the
equivalent form $-y^T A_{\omega^*} y = -\langle \nabla C, y\rangle$,
$A_{\omega^*} y:= \nabla C$, arising naturally in the analysis (and
in practice most easily computable).} Here, we give a particularly
simple derivation\footnote{Essentially a redaction of the argument
of \cite{M}} based on Sylvester's law of inertia.

\subsection{Constrained minimization and Lagrange Multiplier formulation}\label{s:const}

Consider the general constrained minimization problem
\be\label{min}
\min_{X:C(X)=c_0} \mathcal{E}(X),\quad C=(C_1,\dots,C_r):\HM\to \RM^r,
\ee
on a Hilbert space $\HM$. Here, the constraint function $C$ is $C^2$ on $\mathbf{H}$. Also, we assume that $\mathcal{E}$ is differentiable on $\mathrm{Diff}(\mathcal{E})$ a dense subspace of $\mathbf{H}$, and that for every $X\in\mathrm{Diff}(\mathcal{E})$, $\nabla\mathcal{E}(X)$ and $\nabla^2\mathcal{E}(X)$ are well-defined at least on a subset of $\mathrm{Diff}(\mathcal{E})$ that is dense in $\mathbf{H}$.
In addition, we assume
\be\label{h1} \tag{H1}
\hbox{\rm $\nabla C$ is full rank in the vicinity
under consideration.}
\ee
Under assumption \eqref{h1}, the feasible set $\{C(X)=c_0\}$
by the Implicit Function Theorem
defines a local $C^2$ manifold of co-dimension $r$ in the vicinity of any point $X^*$ with
$C(X^*)=c_0$, tangent to $\nabla C(X^*)^\perp$,\footnote{
We use this notation to denote $\{\nabla C_j(X^*):j=1,\dots,r\}^\perp$,
slightly abusing the notation.} on which one may in principle carry out standard Calculus of Variations on local coordinate
charts.

Alternatively, one may solve an optimization problem on an
enlarged flat space, following the method of Lagrange.
Introduce the associated Lagrange multiplier problem
\be\label{lag}
\min_{X\in \HM,\, \omega \in \RM^r}\mathcal{E}_\omega(X)\quad\mbox{where}\quad \mathcal{E}_\omega(X):= \mathcal E(X)+
\omega
\cdot (C(X)-c),
\ee
$\omega\in \RM^r$ are Lagrange multipliers, and the
constrained Hessian
(in $X$ only) is given by:
\be\label{pos}
A_\omega:=\nabla^2 \mathcal{E}_\omega(X^*)=
\nabla^2 \mathcal{E}(X^*)+\omega \cdot \nabla^2 C(X^*).
\ee Critical points
$(X^*,\omega^*)\in\mathrm{Diff}(\mathcal{E})\times\RR^r$ of
\eqref{lag} satisfy $C(X^*)=c_0$ and \be\label{critlag}
\nabla \mathcal{E}(X^*)=-\sum\limits_{j=1}^{r}\omega_j^* \nabla C_j(X^*).
\ee

We recall the following well-known characterizations of stability.

\bpr [First derivative test] \label{1var}
An element of $X^*\in\mathrm{Diff}(\mathcal{E})$ is a critical point of \eqref{min} if and only if there exists $\omega^*$ such that $(X^*,\omega^*)$ is
a critical point of \eqref{lag} for some $\omega^*\in \RM^r$.
\epr
\begin{proof}
Critical points of \eqref{min} satisfy
$C(X^*)=c_0$ and $\langle\nabla \mathcal{E}(X^*),Y\rangle=0$ for $Y\in \nabla C(X^*)^\perp$,
yielding $ \nabla \mathcal{E}(X^*)\in \mathrm{Span}\{\nabla C_j(X^*):j=1,\dots,r\}, $ or \eqref{critlag}.
\end{proof}

\bpr [Second derivative test] \label{2var}
Assume hypothesis \eqref{h1}. Then the following assertions hold true:
\begin{enumerate}
\item[(i)] A critical point $X^*\in\mathrm{Diff}(\mathcal{E})$ is
strictly stable in $\mathrm{Diff}(\mathcal{E})$
(a strict local minimum of \eqref{min}
among competitors in $\mathrm{Diff}(\mathcal{E})$)
if $A_{\omega^*}$ is positive definite on
$\nabla C(X^*)^\perp\cap \mathrm{Diff}(\mathcal{E})$,
\item[(ii)] It is unstable
in $\mathrm{Diff}(\mathcal{E})$,
(not a nonstrict local minimizer of \eqref{min}
among competitors in $\mathrm{Diff}(\mathcal{E})$),
hence in $\HM$,
if $A_{\omega^*}$ is not positive semi-definite on
$\nabla C(X^*)^\perp\cap \mathrm{Diff}(\mathcal{E})$,
\end{enumerate}
\epr

\begin{proof}
Equivalently, $(d^2\mathcal{E}(X(t))/dt^2)|_{t=0}< 0$
(resp. $\not \le 0$)
for any path $X(t)\in\mathrm{Diff}(\mathcal{E})$ with $X(0)=X^*$, $C(X(t))\equiv c_0$.
Taylor expanding, we have that
$ y(t):=X(t)-X_*= tY_1 + t^2Y_2+\mathcal{O}(t^3)$,
where $Y_1\in \nabla C(X^*)^\perp$ and
$\langle \nabla C_j (X^*),Y_2\rangle=- \frac12
\langle Y_1, \nabla^2 C_j(X^*)Y_1\rangle $ for any $j=1,\dots,r$.
Using, \eqref{critlag} we obtain that
$$
\begin{aligned}
\frac12(d^2\mathcal{E}(X(t))/dt^2)|_{t=0}&=
\langle Y_1, \frac12 \nabla^2 \mathcal{E}(X^*)Y_1\rangle+
\langle \nabla \mathcal{E}(X^*),Y_2\rangle\\
&=
\langle Y_1, \frac12 \nabla^2 \mathcal{E}(X^*)Y_1\rangle-\sum\limits_{j=1}^{r}\omega_j^*\
\langle  \nabla C_j(X^*),Y_2\rangle\\
&=
\langle Y_1, \frac12 \nabla^2 \mathcal{E}(X^*)Y_1\rangle-\frac12\sum\limits_{j=1}^{r}\omega_j^*\
\langle  Y_1, \nabla^2 C_j(X^*)Y_1\rangle
\\
&=
\frac12 \langle Y_1, A_{\omega^*} Y_1\rangle,
\end{aligned}
$$
from which both assertions readily follow.
\end{proof}

\br
Note that the second variation conditions for the Lagrange multiplier
problem are with respect to the constrained
Hessian, and along the tangent feasible space $\nabla C(X^*)^\perp$
only.
\er

\br\label{strongrmk} The conclusions of this section are strongest,
evidently, when $\mathrm{Diff}(\mathcal{E})=\HM$, in which case
strict stability implies variation in energy controls variation in
$\HM$, or \be\label{2control} \hbox{\rm $\|X_1-X^*\|^2_{\HM}\lesssim
\mathcal{E}(X_1)- \mathcal{E}(X^*)$ for $\|X_1-X^*\|_{\HM}$
sufficiently small.} \ee This holds, for example, for the
Cahn--Hilliard and generalized KdV energies considered in Sections
\ref{s:CH} and \ref{s:gKdV}. However, the Keller--Segel and
viscoelastic energies considered in Sections \ref{s:KellerSiegel}
and \ref{s:visco} are not even continuous on $\HM$, hence the more
general phrasing here; see Remarks \ref{diffrmk} and \ref{contrmk}.
\er

\subsection{Sylvester's Law and relative signature}\label{s:index}
Let $Q:\dom(Q)\subset\HM\to\HM$ be a self-adjoint operator on $\HM$.
For $\alpha\in\{-,0,+\}$ we denote by $\mathcal{S}_\alpha(Q)$ the
set of all nontrivial {\em finite} dimensional subspaces $V\subset\dom(Q)$ such
that $Q_{|V}$ is negative definite, identically zero, and positive
definite, respectively. Define\footnote{Note that in the case of $\sigma_0$, we require $Q$, not the quadratic form $\langle \cdot,Q\cdot\rangle$, to be identically zero on $V$ so that $\sigma_0(Q)$ is the dimension of the kernel of $Q$.}
\be\label{def-signature1}\sigma_\alpha(Q)=
\left\{\begin{array}{ll}
\sup\{\dim V: V\in\mathcal{S}_\alpha(Q)\},&\mbox{if}\; \mathcal{S}_\alpha(Q)\ne\emptyset\\
0,&\mbox{if}\; \mathcal{S}_\alpha(Q)=\emptyset\end{array}\right.\in
\ZZ_+\cup\{\infty\},\quad\alpha\in\{-,0,+\}. \ee We introduce the
signature of the self-adjoint operator $Q$ by
\be\label{def-signature2}
\sigma(Q):=(\sigma_-(Q),\sigma_0(Q),\sigma_+(Q)). \ee Recall now the
following fundamental observation.

\bpr[Sylvester's Law of Inertia]\label{Syl}
The signature $\sigma$ is invariant under coordinate changes,
i.e., $\sigma(Q)=\sigma (S^*QS)$ for any bounded invertible coordinate
change $S$.
\epr

\begin{proof}
Invariance follows immediately since the change of variables $S$ is invertible and
$$V\in \mathcal{S}_\alpha(Q)\quad\mbox{implies that}\quad S^{-1}(V)\in\mathcal{S}_\alpha(S^*QS) \quad\mbox{and}\quad \dim V=\dim S^{-1}(V);$$
$$U\in \mathcal{S}_\alpha(S^*QS)\quad\mbox{implies that}\quad  S(U)\in\mathcal{S}_\alpha(Q)\quad\mbox{and}\quad \dim U=\dim S(U).$$
\end{proof}
In the sequel, we will need the following proposition:
\bpr\label{p2.5} Assume $Q:\dom(Q)\subset\HM\to\HM$ is a
self-adjoint operator. The the following assertions hold true:
\begin{enumerate}
\item[(i)] If $B$ is a self-adjoint, invertible, $s\times s$ matrix and $R:\dom(Q)\otimes\CC^s\to\HM\otimes\CC^s$ is defined by $R=\bp
Q&0\\0&B\ep$, then
\be\label{sigma-sum}
\sigma_\alpha(R)=\sigma_\alpha(Q)+\sigma_\alpha(B),\quad\alpha\in\{-,0,+\}.\ee
\item[(ii)] If $\sigma_\alpha(Q)$, $\alpha\in\{-,0,+\}$, is finite then it is equal to the dimensions of the stable, zero, and
unstable eigenspaces of $Q$, respectively;
\end{enumerate}
\epr
\begin{proof} (i) First we note that our assertion is true if $B$ is a nonzero real
number. By induction, we conclude that
\eqref{sigma-sum} is also true provided the
matrix $B$ is a diagonal matrix with nonzero real entries.

Next, we look at the general case when $B$ is a general
self-adjoint matrix. Then, there exists $P_B$ an
orthogonal matrix and $D$ a diagonal matrix such that
$B=P_B^TDP_B$. Since any orthogonal matrix is unitary we have that
$P_B^T=P_B^*$, $\bp I_{\HM}&0\\0&P_B\ep$ is invertible and
$$R=\bp I_{\HM}&0\\0&P_B\ep^*\bp A&0\\0&D\ep \bp I_{\HM}&0\\0&P_B\ep.$$
Assertion (i) follows from Proposition~\ref{Syl} and the observation that \eqref{sigma-sum} is true for diagonal matrices $B$ with nonzero real entries.

We note that assertion (ii) is trivial for $\alpha=0$. Assume that
$\sigma_+(Q)$ is finite. Then, there exists a subspace
$\tilde{\HM}$, a self-adjoint, negative-definite operator
$\tilde{Q}:\dom(\tilde{Q})\subset\tilde{\HM}\to\tilde{\HM}$ and $D$
a diagonal $s\times s$ matrix with positive real entries, where $s$
is equal to the dimension of the unstable space of $Q$, so that
$Q=\bp \tilde{Q}&0\\0&D\ep$. Since $\sigma_+(\tilde{Q})=0$ and
$\sigma_+(D)=s$, our conclusion follows immediately from (i).
\end{proof}

Assume now

\be\label{h2}\tag{H2}
\hbox{\rm Near $(X^*, \omega^*)$, there is a $C^1$ family $(X, c)(\omega)$
of critical points of \eqref{lag}.}
\ee

\noindent This follows, for example, by the Implicit Function Theorem,
applied to \eqref{critlag} if $\ker A_{\omega^*}$ is trivial
and also $A_{\omega^*}$ is uniformly invertible.
Another situation in which \eqref{h2} follows, of particular interest in the
applications considered here, is when $\ker A_{\omega^*}$ is due to
group symmetries of the problem leaving both energy $\mathcal{E}$ and
constraints $C$ invariant.
and also $A_{\omega^*|_{(\ker A_{\omega^*})^\perp}}$ is uniformly invertible,
i.e., $A_{\omega^*}$ has a spectral gap.
In general, it may be obtained as part of an existence theory
for Euler-Lagrange equation \eqref{critlag} carried out by other means.
We note that \eqref{h2} in any case implies also
\be\label{perprel}
\ker A _{\omega^*}\perp \nabla C(X^*),
\ee
(by Fredholm alternative), as holds automatically when
$\ker A_{\omega_*}$ is generated by group invariances.

Assume also
\be\label{h3}\tag{H3}
\hbox{\rm The matrix  $(\partial c/\partial \omega)(\omega^*)$ is invertible.}
\ee

Differentiating the Euler--Lagrange relation \eqref{critlag} with
respect to $\omega$, we obtain the key relation
\be\label{orthrel}
A_{\omega^*} (\partial X/\partial \omega_j)(\omega^*)= - \nabla C_j(X^*)\quad\mbox{for all}\quad j=1,\dots,r
\ee
and its consequence
\be\label{hessrel}
(\partial c_k/\partial \omega_j)(\omega^*)=
-\langle(\partial X/\partial \omega_j)(\omega^*), A_{\omega^*}
(\partial X/\partial \omega_k)(\omega^*)\rangle\quad\mbox{for all}\quad j,k=1,\dots,r,
\ee
from which we see that $(\partial c/\partial \omega)(\omega^*)$ is
a symmetric form. Our main result is the following identity relating the signatures
of $A_{\omega^*}|_{\nabla C(X^*)^\perp}$, $A_{\omega^*}$,
and $(\partial c/\partial \omega)(\omega^*)$.

\bt\label{main}
Assuming \eqref{h1}--\eqref{h3},
\be\label{sigrel}
\sigma(A_{\omega^*})=\sigma\Bigl(A_{\omega^*}|_{\nabla C(X^*)^\perp}\Bigr )+\sigma\Bigl(-\partial c/\partial \omega)(\omega^*)\Bigr).
\ee
\et

\begin{proof}
To prove the theorem we use  Sylvester's Law. First, we note that
from \eqref{h3}, \eqref{orthrel} and \eqref{hessrel} it follows that
\be\label{sp-sum} \mathbf{H}=\nabla
C(X^*)^\perp\oplus\mathrm{Span}\{(\partial
X/\partial\omega_j)(\omega^*):j=1,\dots,r\}. \ee  Next, we denote $\Pi_{\nabla C(X^*)^\perp}$ and
$\Pi_{\mathrm{Span}\{(\partial
X/\partial\omega_j)(\omega^*):j=1,\dots,r\}}$  the orthogonal
projections onto $\nabla C(X^*)^\perp$ and $\mathrm{Span}\{(\partial
X/\partial\omega_j)(\omega^*):j=1,\dots,r\}$, respectively. Using
\eqref{sp-sum} we infer that $S: \nabla
C(X^*)^\perp\times\mathrm{Span}\{(\partial
X/\partial\omega_j)(\omega^*):j=1,\dots,r\}\to\mathbf{H}$ defined by
\be\label{S} S\bp h_1\\h_2\ep=h_1+h_2 \ee is bounded, invertible
with bounded inverse. Moreover, from \eqref{orthrel} and
\eqref{hessrel} we obtain that \be\label{inv1} \Pi_{\nabla
C(X^*)^\perp}A_{\omega^*}\Pi_{\mathrm{Span}\{(\partial
X/\partial\omega_j)(\omega^*):j=1,\dots,r\}}=0. \ee Since
$A_{\omega^*}$ and the orthogonal projections are self-adjoint
operators, passing to the adjoint in \eqref{inv1} yields
\be\label{inv2} \Pi_{\mathrm{Span}\{(\partial
X/\partial\omega_j)(\omega^*):j=1,\dots,r\}}A_{\omega^*}\Pi_{\nabla
C(X^*)^\perp}=0. \ee Finally, from \eqref{S}--\eqref{inv2} we
conclude that \be\label{rep} S^*A_{\omega^*} S= \bp
 \Pi_{\nabla C(X^*)^\perp} A_{\omega^*}|_{\nabla C(X^*)^\perp} & 0\\
0 & \Pi_{\mathrm{Span}\{(\partial
X/\partial\omega_j)(\omega^*):j=1,\dots,r\}}A_{\omega^*}|_{\mathrm{Span}\{(\partial
X/\partial\omega_j)(\omega^*):j=1,\dots,r\}} \ep. \ee Next, we note
that from \eqref{orthrel} and \eqref{hessrel} it follows that
\be\label{s2as2} \Pi_{\mathrm{Span}\{(\partial
X/\partial\omega_j)(\omega^*):j=1,\dots,r\}}A_{\omega^*}h=-\sum\limits_{j,k=1}^{r}
(\partial c_k/\partial\omega_j)(\omega^*)\langle h,(\partial
X/\partial\omega_k)(\omega^*)\rangle (\partial
X/\partial\omega_j)(\omega^*), \ee which implies that
\be\label{c-s2as2} \sigma\Bigl(\Pi_{\mathrm{Span}\{(\partial
X/\partial\omega_j)(\omega^*):j=1,\dots,r\}}A_{\omega^*}|_{\mathrm{Span}\{(\partial
X/\partial\omega_j)(\omega^*):j=1,\dots,r\}} \Bigr)= \sigma
\Bigl(-(\partial c/\partial \omega)(\omega^*)\Bigr). \ee Since
$\sigma( \Pi_{\nabla C(X^*)^\perp} A_{\omega^*}|_{\nabla
C(X^*)^\perp})=\sigma(A_{\omega^*}|_{\nabla C(X^*)^\perp})$, the
theorem follows shortly from Proposition~\ref{p2.5}(i), \eqref{rep}
and \eqref{c-s2as2}.
\end{proof}

\subsection{Variational stability}\label{s:var}
Denoting $\ell:=\dim \ker A_{\omega^*}$, we strengthen \eqref{h2} to
\be\label{h2'}\tag{H2'}
\hbox{\rm Near $(X^*, \omega^*)$, there is a $C^1$ family $(X, c)(\omega,s)$
of critical points of \eqref{lag},
$s\in \RM^\ell$.}
\ee
\noindent With \eqref{h3}, this implies that, for fixed $c$, the
set of minimizers near $X^*$
forms a $C^1$ $\ell$-dimensional family
\be\label{minfam}
\{X(\omega(c),s): \, s\in \RM^\ell\}.
\ee
Differentiating \eqref{minfam} with respect to $s$,
noting that $C(X(\omega(c),s))\equiv \const$,
and recalling \eqref{critlag}, we find that
\be\label{constcrit}
\hbox{\rm
$\mathcal{E}(X(\omega(c),s))\equiv \const$ along the family of nearby critical points.}
\ee

\begin{definition}\label{stabdef}
A critical point $X^*$ of \eqref{min} is
locally strictly orbitally stable in $\hat \HM\subset \HM$ if, for $X\in \hat \HM$
in a sufficiently small $\HM$-neighborhood of $X^*$,
$\mathcal{E}(X)\geq \mathcal{E}(X^*)$,
with equality only if $X$ lies on the set of nearby critical points,
and orbitally stable in $\hat \HM\subset \HM$
if,
for $X\in \hat \HM$ in a sufficiently small $\HM$-neighborhood of $X^*$,
$\mathcal{E}(X)\geq \mathcal{E}(X^*)$.
\end{definition}

We add the further hypothesis,
denoting $\HM$-distance from $v\in \HM$ to a set $S\subset \HM$
by $d_{\HM}(v,S)$:
\be\label{h4}\tag{H4}
\hbox{\rm
$\sigma_-(A_{\omega^*}|_{\nabla C(X^*)^\perp} )=0$
$\Rightarrow$ for $K>0$, all $v\in \nabla C^\perp(X^*)$,
 $d_{\HM}(v, \ker A_{\omega^*|_{\nabla C^\perp(X^*)}})
\leq K\langle v, A_{\omega^*}v\rangle_{\HM}$.
}
\ee

\bc\label{stabcor}
Assuming
\eqref{h1}, \eqref{h2'}, \eqref{h3},
and \eqref{h4},
a critical point $X^*$ of \eqref{min}
is strictly orbitally stable in $\mathrm{Diff}(\mathcal{E})$
if and only if
\be\label{sorbvar}
 \sigma_- (A_{\omega^*}) = \sigma_-\Bigl( (-\partial c/\partial \omega)(\omega^*)\Bigr)
\ee
and $\dim \ker A_{\omega_*}|_{\nabla C^\perp(X^*)}=\ell$,
i.e., family  \eqref{h2} is transversal,
and orbitally stable in $\mathrm{Diff}(\mathcal{E})$ only if
\be\label{orbvar}
 \sigma_- (A_{\omega^*}) = \sigma_-\Bigl( (-\partial c/\partial \omega)(\omega^*)\Bigr),
\ee
where $\omega^*$ is the (unique) Lagrange multiplier associated with $X^*$.
\ec

\begin{proof}
Immediate, from \eqref{sigrel}, \eqref{constcrit},
and \eqref{h4},
together with a second-order Taylor expansion
like that of Proposition \ref{2var}
along perpendiculars to the manifold \eqref{h2} where $C\equiv c_*$.
\end{proof}

\br\label{4rmk}
In the PDE context, Assumption \eqref{h4}
effectively limits the regularity of the space
$\HM$.  More generally, the $\HM$-norm cannot be stronger
than any norm under which $A_{\omega_*}$ is bounded,
or \eqref{h4} would fail.
If $A_{\omega_*}$ is bounded on $\HM$
(in particular, if $\mathrm{Diff}(\mathcal{E})=\HM$),
then $\HM$ is determined.
In applications, \eqref{h4} typically follows by coercivity
of $A_{\omega_*}$, in the form of a G\r{a}rding inequality,
together with the presence of a spectral gap.
For all of the (PDE) examples of this paper, there is a unique space $\HM$
on which $A_{\omega_*}$ is both bounded and coercive in the
strong sense of \eqref{h4}.
\er

\section{Time-evolutionary stability}\label{s:evolution}
Corollary \ref{stabcor} yields
stability results for time-evolutionary systems
\be\label{te}
(d/dt)X=\mathcal {F}(X)
\ee
with variational structure,
in particular, for {\it energy conserving} (e.g., generalized Hamiltonian),
and {\it energy dissipating} (e.g., quasi-gradient type) systems,
as we now describe.

For this section, it will be convenient to introduce
a graded sequence of Hilbert spaces $\mathbf{H}_1\subset \mathbf{H} \subset \mathbf{H}_2$,
with associated inner products $\langle \cdot, \cdot\rangle_{\mathbf{H}_j}$,
$\langle \cdot, \cdot\rangle_{\mathbf{H}}$ of varying strength,
and gradients $\nabla_{\mathbf{H}_j}$, satsifying:
\be\label{b1}\tag{B1}
\hbox{\rm
For $X\in\mathbf{H}_1$, $\nabla_{\mathbf{H}_2}$
and $\nabla_{\mathbf{H}_2}^2$ of
$\mathcal{E}$ and $C$ are well-defined.
}
\ee
In particular, this implies that $\HM_1\subset {\rm Diff}(\mathcal{E})$.

\subsection{
Generalized Hamiltonian systems with constraints}\label{s:genham}
We first recall briefly the context considered in \cite{GSS}
of a generalized Hamiltonian system
\ba\label{e:genham}
(d/dt) X&= J(X) \nabla_{\mathbf{H}_2} \mathcal{E}(X),
\quad t\ge 0,\\
X(0)&\in \HM, \ea
$J=-J^*$,
which has additional conserved quantities
\be\label{C}
\hbox{\rm $C(X(t))\equiv c_0\in \RM^r$, i.e., $\langle J(X(t))\nabla
C_j(X(t)), \nabla \mathcal{E}(X(t))\rangle\equiv 0$,}\;
j=1,\dots,r,\;t\geq 0 \ee
besides the Hamiltonian
$\mathcal{E}(X(t))$, which is always conserved, by skew-symmetric
structure
\be\label{skw-sym} \langle \nabla
\mathcal{E}(X(t)),J(X(t))\nabla\mathcal{E}(X(t))\rangle=0. \ee
Here $J=-J^*$, is assumed to be a closed, densely
defined linear operator with $X$-independent domain.

For such systems, constraints may be divided
into two rather different types. The first type, familiar from the
classical mechanical case, are constraints $C_j:\mathbf{H}\to \RM$
for which \be\label{first_class} g_j(X):=J(X)\nabla C_j(X) \neq
0,\quad\mbox{for all}\quad j=1,\dots,r. \ee These are associated
with group symmetries $G_j(\theta,\cdot)$ leaving the Hamiltonian
$\mathcal{E}$, and thus the flow of \eqref{e:genham}, invariant,
generated by the $g_j$, through \be\label{genode}
(d/d\theta)G_j(\theta;X)=g_j(G_j(\theta;X)), \qquad G_j(0;X)=X,\quad
j=1,\dots,r. \ee
When $J$ is one-to-one,\footnote{
Note: by skew-symmetry, $J$ is one-to-one if it is onto,
as sometimes assumed \cite{GSS}.
We remark that $J$ is also Fredholm in the typical situation where we have
such non-Noetherian constraints, for example in the case of NLS ($J=\mathrm{i}$)
or gKdV ($J=\partial_x$) on bounded periodic domains, in which
case $J$ is one-to-one if and only if it is onto.
}
these are the only type of
conserved quantities that can occur, and every group invariance
determines a conserved quantity, following Noether's Theorem,
through the relation $\nabla C_j(X):=J^{-1}g_j(X)$, where
$g_j(X):=G_j'(0;X)$, $j=1,\dots,r$.

When $J$ is not one-to-one, however, as for example in the KdV and
Boussinesq cases considered in \cite{GSS,Z2,BJK}, for which
$J=\partial_x$, conserved quantities may or may not correspond to
symmetries. For related discussion, see \cite{Z2}. In this case, we
identify a second class of constraints, closer to the themes of this
paper, for which \be\label{second_class} J\nabla C_j \equiv
0\quad\mbox{for all}\quad j=1,\dots,r. \ee These are not associated
to a nontrivial group symmetry, since $g_j=J\nabla C_j$,
$j=1,\dots,r$ generates the identity.

For example, in the case considered in \cite{GSS} of KdV on the whole line,
$r=1$, and the constraint is associated to a group invariance
$G(s)$ consisting of spatial translation by distance $s$.
For the generalized KdV case considered in \cite{BJK} on bounded periodic
domains, $r=2$, only one of the two constraints
corresponding to translation invariance and the
other corresponding to conservation of mass (integral over
one period).

Denote the (possibly trivial) symmetries generated by constraints
$C_j$ through \eqref{genode} as $G_j(\theta;X) $, $j=1,\dots,r$.

Following \cite{GSS}, we assume that \eqref{e:genham} has a well-defined
solution operator in $\HM$ conserving the formal invariants $C$ and
$\mathcal{E}$:
\be\label{b2}\tag{B2}
\hbox{\rm
Equation \eqref{e:genham} admits a (weak) solution $X(t)\in \HM$
continuing so long
}
\ee
$$
\hbox{\rm
as $\|X(t)\|_{\HM}<+\infty$,
with $C(X(t))$ and $\mathcal{E}(X(t))$ identically constant.
}
$$

\begin{definition}\label{travdef}
A generalized traveling-wave solution of \eqref{e:genham} is
a solution whose trajectory is contained in the orbit of its
initial point $X(0)$ under the action of the invariance group
generated by $\{G_j:j=1,\dots,r\}$.
\end{definition}

\bl\label{hamcrit}
A solution emanating from $X(0)\in \mathrm{Diff}(\mathcal{E})$
is a generalized traveling wave
if and only if $X(0)$ is a critical point of the constrained minimization
problem \eqref{min}, in which case $X(0)$ is a stationary
solution of the shifted equation $(d/dt)X=J(X)\nabla \mathcal{E}_{\omega^*}(X)$,
with associated linearized operator
\be\label{critfactor}
L_{\omega^*}(X(0)):=d\mathcal{F}_{\omega^*}(X(0))=J(X(0))A_{\omega^*}.
\ee
\el

\begin{proof}
By group invariance of \eqref{e:genham}, it is equivalent that
$X'(0)=J(X(0))\nabla \mathcal{E}_{\omega^*}(X(0))$ lie in the
tangent space $\Span \{J(X(0))\nabla C_j(X(0)):j=1,\dots,r\}$ of the invariance group. Since
$\ker J(X(0))\subset \Span\{\nabla C_j(X(0)):j=1,\dots,r\}$, this is equivalent to
\eqref{critlag}, which in turn is equivalent to stationarity as a
solution of the shifted equation. Writing, equivalently,
$\mathcal{F}=J \nabla \mathcal{E}_{\omega^*}$, differentiating in
$X$, and using
\eqref{critlag},
we obtain \eqref{critfactor}.
\end{proof}

\begin{definition}\label{tstabdef}
A generalized traveling wave $X(t)\in\mathbf{H}_1$ is $\mathbf{H}_1\to\mathbf{H}$
orbitally stable if
for every $\eps>0$, there is $\delta>0$ such
that, for  $Y(0)$ within $\delta$ distance in $\mathbf{H}_1$ of $X(0)$, $Y(t)$
exists for all $t\ge 0$ and
remains within $\eps$ distance in $\mathbf{H}$ from the orbit of $X(0)$ under
the group of invariances generated by $\{G_j:j=1,\dots,r\}$.
\end{definition}

Given a generalized traveling wave solution $X(t)$, we assume
further: \be\label{c1}\tag{C1} \hbox{\rm $\ker L_{\omega^*}(X(0))\subset
\Span \{J(X(0))\nabla C_j(X(0)):j=1,\dots,r\}$ (transversality), }
\ee that is, that all neutral directions are accounted for by group
symmetry.

\bpr\label{gssprop} Assuming \eqref{h1}, \eqref{h2'},\eqref{h3}--\eqref{h4},
\eqref{b1}-\eqref{b2}, \eqref{c1}, a generalized traveling-wave
solution $X(t)$ of \eqref{e:genham} is
$\mathbf{H}\to\mathbf{H}$
orbitally stable if $X(0)\in\mathbf{H}_1$ is a strict local orbital
minimizer of \eqref{min}, or, equivalently, the signature of
$(\partial c/\partial \omega)(\omega^*)$ satisfies condition
\eqref{sorbvar},
and ${\rm Diff}(\mathcal{E})=\HM$. \epr

\begin{proof}
By \eqref{h1}, the level sets of $C$ smoothly foliate
$B(X(0),\delta)$, whence it is sufficient to consider perturbations
such that $C(Y(0))=C(X(0))$. As
$\mathcal{E}(Y(t))-\mathcal{E}(X(t))$ is constant, and, by strict
orbital minimality,
majorizes
the $\|\cdot\|_\mathbf{H}^2$ distance of $Y(t)$ from the manifold of equilibrium
solutions through $X(0)$, which, by \eqref{c1}, is the group orbit of
$X(0)$,\footnote{
Here, we are using group invariance to conclude from local majorization
a global property as well;
this, precisely, is the step where we require \eqref{c1}.
}
we obtain global existence by \eqref{b1}
and continuation,
and $\eps$-boundedness by the assumed continuity of $\mathcal{E}$ with
respect to $\|\cdot\|_\mathbf{H}$.
\end{proof}

\br\label{phasermk}
The notion of orbital stability defined above does not include
phase-asymptotic stability; that is, it does not assert that $Y(t)$
stays close to $X(t)$.  It is not difficult to see by specific example
(for instance, the different-speed traveling-waves of
 \cite{GSS}), that this is all that we can say.
\er

\br\label{genrmk} Condition \eqref{c1} could be weakened to the
assumptions that for $c$ near $c^*$ the manifold of critical points
\eqref{minfam} extends globally as the orbit under the group
symmetries generated by $\{G_j:j=1,\dots,r\}$ of a (for example,
compact) manifold of strict local orbital minimizers of \eqref{min}
on which $A_\omega$ has a uniform spectral gap, a possibility that
does not seem up to now to have been discussed.
\er

\br\label{oneway}
The ``one-way'' stability result of Proposition \ref{gssprop}
recovers the result of \cite{GSS} for solitary waves of KdV;
for $J$ onto (one-to-one), a converse is given in \cite{GSS}
for $r=1$.
In general, time-evolutionary stability does not imply constrained
variational stability but only certain parity information on the
number of unstable eigenvalues.
For complete stability information, a more detailed analysis must
be carried out, such as the Krein signature analysis in \cite{BJK}.
\er

\subsection{Quasi-gradient systems with constraints}\label{s:gengrad}

We now turn to the situation of our main interest,
of {\it quasi-gradient systems}
\be\label{gengrad}
(d/dt)X= -M(X) \nabla \mathcal{E}(X),
\qquad
\frac12(M+M^*)\ge 0,
\ee
with constants of motion $ C(X(t))\equiv c_0\in \RM^r$ for which
\begin{equation}\label{E-M}
(d/dt) \mathcal{E}(X(t))=
\langle\nabla \mathcal{E}(X(t)),  M((X(t))  \nabla \mathcal{E}(X(t))\rangle
\le 0,\end{equation}
guaranteed by
\be\label{M_class}
M\nabla C_j \equiv 0\quad\mbox{for all}\quad j=1,\dots,r,
\ee
similarly as in \eqref{second_class}.
Recall, for $M$ semidefinite, $\ker M=\ker M^*
\subset \ker \frac12(M+M^*)$,
by consideration of the semi-definite quadratic form $\frac12(M+M^*)$,
so that \eqref{M_class} implies also $\langle\nabla C_j,M\nabla\mathcal{E}\rangle\equiv 0$ for all $j=1,\dots,r$, giving $(d/dt)C(X(t))\equiv 0$.

For the Keller--Segel (chemotaxis) model considered in \cite{PS1,PS2},
$r=1$, with the single constraint corresponding to conservation
of mass (total population).
For the equations of viscoelasticity with strain-gradient effects on
bounded periodic domains, considered in \cite{OZ1,Y,BYZ},
$r=2d=2,4,6$, where $d=1,2,3$ is
the dimension of allowed deformation directions,
with the $2d$ constraints
corresponding to conservation of mass
in deformation velocity and gradient coordinates.

\bl\label{gradcrit}
$X(t)\equiv X(0)\in {\rm Diff}(\mathcal{E})$
is a stationary point of \eqref{gengrad}
if and only if it is a critical point of the constrained minimization
problem \eqref{min}, in which case the associated linearized operator is
\be\label{gcritfactor}
L_{\omega^*}(X(0)):=d\mathcal{F}_{\omega^*}(X(0))=-M(X(0))A_{\omega^*}.
\ee
\el

\begin{proof}
Essentially identical to the proof of Lemma \ref{hamcrit}.
\end{proof}

For the applications we have in mind, we find it necessary to
substitute for \eqref{b2} the milder assumption:
\be\label{b2'}\tag{B2'}
\hbox{\rm
For data $X(0)\in \HM_1$ near an equilibrium solution $X^*$,
\eqref{te} admits a strong
}
\ee
$$
\hbox{\rm
solution $X(t)\in \HM_1$, with $(d/dt)X(t)=\mathcal{F}(X(t)) \in\mathbf{H}_2$.
Moreover, this solution
}
$$
$$
\hbox{\rm
continues so long as
$X(t)$ remains sufficiently close in $\HM$ to an equilibrium.
}
$$

\br\label{b2'rmk} Condition \eqref{b2'} is satisfied for all of the
example quasigradient systems considered in this paper,\footnote{
The continuation property follows by nonlinear damping (energy)
estimates as introducted in \cite{Z4,Z1}, which show that higher
Sobolev norms are exponentially slaved to lower ones.. } whereas
\eqref{b2} is satisfied only for the Cahn--Hilliard equations
considered in Section \ref{s:CH}. \er

\subsubsection{Orbital stability}\label{s:orbital}
Assume analogously to \eqref{c1}:
\be\label{c1'}\tag{C1'}
\hbox{\rm
$\ker L_{\omega^*}$ is generated by group symmetries of \eqref{gengrad}.
}
\ee

\begin{definition}\label{astabdef}
A stationary solution $X(t)\equiv X(0)\in\mathbf{H}_1$ of
\eqref{gengrad} is orbitally stable if for every $\eps>0$, there is
$\delta>0$ such that, for  $Y(0)$ within $\delta$ distance in
$\mathbf{H}_1$ of $X(0)$, $Y(t)$ exists for all $t\ge 0$ and remains
within $\eps$ distance in $\mathbf{H}$ of the manifold
\eqref{minfam} of stationary solutions (critical points of
\eqref{min}) near $X(0)$.
It is asymptotically orbitally stable if it is orbitally stable and,
for $\|Y(0)-X(0)\|_{\mathbf{H}_1}$ sufficiently small,
$Y(t)$ converges in $\mathbf{H}$ to the
manifold \eqref{minfam}
of stationary solutions (critical points of \eqref{min}) near $X(0)$
having the fixed value $c=C(Y(0))$.
\end{definition}
\bpr\label{grad_gssprop} Assuming \eqref{h1}, \eqref{h2'},
\eqref{h3}--\eqref{h4}, \eqref{b1}-\eqref{b2'}, \eqref{c1'}, a stationary
solution $X(t)\equiv X(0)$ of \eqref{gengrad} is
$\mathbf{H}_1\to\mathbf{H}$ orbitally stable if $X(0)\in
\mathbf{H}_1$ is a strict local orbital minimizer of \eqref{min},
or, equivalently, the signature of $(\partial c/\partial
\omega)(\omega^*)$ satisfies condition \eqref{sorbvar}. \epr

\begin{proof}
Identical to that of Proposition \ref{gssprop}.
\end{proof}

\subsubsection{Asymptotic orbital stability: finite-dimensional case}\label{s:asymptotic}
Restricting for simplicity to finite dimensions, take without loss of generality
$\mathbf{H}_2=\mathbf{H}=\mathbf{H}_1$, by equivalence of finite-dimensional norms.

Now, make the additional assumption:
\be\label{d1}\tag{D1}
\hbox{\rm
Constant-energy solutions
$\mathcal{E}(X(t))\equiv \mathcal{E}(X(0))$
of \eqref{gengrad}
are stationary: $X(t)\equiv X(0)$.
}
\ee

\bl\label{noper}
A sufficient condition for \eqref{d1} is
$M=M^*$, or, more generally,.
\be\label{d1'}\tag{D1'}
\hbox{\rm
$\ker M=\ker \frac12(M+M^*)$,
}
\ee
\el

\begin{proof} Using \eqref{E-M} we infer
\begin{align*}
(d/dt)\mathcal{E}(X(t))|_{t=0}&=-\langle\nabla \mathcal{E}(X(0)) , M(X(0)) \nabla \mathcal {E}(X(0))\rangle\\
&=
-\frac12\langle\nabla \mathcal{E}(X(0)), (M+M^*)(X(0)) \nabla \mathcal {E}(X(0)) \rangle
\end{align*}
vanishes if and only if $\nabla \mathcal{E}(X(0)) $ lies in
$\ker\Bigl( \frac12(M+M^*)(X(0))\Bigr)=\ker M(X(0))$, whence $$\mathcal{F}(X(0))=
-M(X(0)) \nabla \mathcal{E}(X(0)) =0.$$
Noting, finally, that $\mathcal{F}_{\omega^*}\equiv \mathcal{F}$,
we are done.
\end{proof}

\br
As \eqref{d1'} shows, \eqref{d1} is typical of gradient systems.
That is, gradient systems typically
do not possess nontrivial constant-energy solutions
such as time-periodic solutions
or generalized traveling-waves, in contrast with the situation
of the Hamiltonian case.
\er

\bpr\label{lyap_prop} In finite dimensions, assuming \eqref{h1},
\eqref{h2'}, \eqref{h3},
\eqref{c1'}, and
\eqref{d1}, a stationary solution $X(t)\equiv X(0)$ of
\eqref{gengrad} is $\mathbf{H}\to \mathbf{H}$ asymptotically
orbitally stable if $X(0)\in \mathbf{H}$ is a strict local orbital
minimizer of \eqref{min}, or, equivalently, the signature of
$(\partial c/\partial \omega)(\omega^*)$ satisfies condition
\eqref{sorbvar}. \epr

\begin{proof}
First note that \eqref{h4} and \eqref{b1}--\eqref{b2'} hold always
in finite dimensions.
By orbital stability, $Y(t)$ remains in a compact neighborhood of
the orbit of $X(0)$ under the group symmetries of \eqref{gengrad},
whence, transporting back by group symmetry to a neighborhood of
$X(0)$, we obtain an $\omega$-limit set which is invariant up to
group symmetry under the flow of \eqref{gengrad}.
Moreover, by nonincreasing of $\mathcal{E}(Y(t))$, and local lower-boundedness
of $\mathcal{E}$, we find that $\mathcal{E}$ must be constant on the
$\omega$-limit set.
Combining these properties, we find that points of the $\omega$-limit
set must lie on constant-energy solutions, which, by \eqref{d1},
are stationary points of \eqref{gengrad}.  Noting that, up to
group symmetry, $Y(t)$ approaches its
$\omega$-limit set by definition, we are done.
\end{proof}

\begin{definition}\label{gradstabdef}
A stationary solution $X(t)\equiv X(0)\in
\mathbf{H}_1$ of \eqref{gengrad}
is $\mathbf{H}_1\to \mathbf{H}$ stable if for every $\eps>0$, there is $\delta>0$ such
that, for  $Y(0)$ within $\delta$ distance in $\mathbf{H}_1$ of $X(0)$, $Y(t)$
exists for all $t\ge 0$ and
remains within $\eps$ distance in $\mathbf{H}$ of $X(0)$.
It is $\mathbf{H}_1\to \mathbf{H}$
phase-asymptotically orbitally stable if it is stable and,
for $\|Y(0)-X(0)\|_{\mathbf{H}_1}\le \delta$, $Y(t)$ converges
time-asymptotically
in $\|\cdot\|_\mathbf{H}$ to a point $X^\sharp(Y(0))$ on the manifold \eqref{minfam}
of stationary solutions (critical points of \eqref{min}) near $X(0)$.
\end{definition}


\bpr\label{gradstab} In finite dimensions, assuming \eqref{h1},
\eqref{h2'},\eqref{h3},
\eqref{d1'}, a
stationary solution $X(t)\equiv X(0)\in \mathbf{H}$ of
\eqref{gengrad} is $\mathbf{H}\to\mathbf{H}$ stable only if $X(0)$
is a nonstrict local minimizer of \eqref{min}, or, equivalently, the
signature of $(\partial c/\partial \omega)(\omega^*)$ satisfies
condition \eqref{orbvar}. It is phase-asymptotically orbitally
stable, indeed, time-exponentially convergent, if $X(0)$ is a strict
local minimizer of \eqref{min}, or, equivalently, the signature of
$(\partial c/\partial \omega)(\omega^*)$ satisfies condition
\eqref{sorbvar}, \epr

\begin{proof}
Again, note that \eqref{h4} and \eqref{b1}-\eqref{b2'} hold always in finite dimensions.
(i) If $X(0)$ is not a minimizer, then there is a nearby point $Y(0)$
with $C(Y(0))=C(X(0))=c^*$ and
$\mathcal{E}(Y(0))< \mathcal{E}(X(0))$.
As $C(Y(t))\equiv C(Y(0))$, the only stationary points to which $Y(t)$ could
converge are on the manifold \eqref{minfam} with $|s|$ small and $c=c^*$,
along which $\mathcal{E}\equiv \mathcal{E}(X(0))$.
But, then, by continuity of $\mathcal{E}$ in $\mathbf{H}$, $\mathcal{E}(Y(t))$
would converge to $\mathcal{E}(X(0))$, a contradiction.  This
establishes the first claim.

(ii)
By \eqref{h1}, the level sets of $C$ smoothly foliate $B(X(0),\delta)$,
whence it is sufficient to consider perturbations such that
$C(Y(0))=C(X(0))$.
If $X(0)$ is a strict local minimizer, then,
in $B(X(0), \eps)$,  $\mathcal{E}(Y(t))$
majorizes the $\|\cdot\|$-distance from $Y(t)$
to the manifold of stationary solutions through $X(0)$.
Moreover,
$$
(d/dt)\mathcal{E}(Y(t))=
\langle \nabla \mathcal{E}(Y(t)), M(Y(t))\nabla \mathcal{E}(Y(t))\rangle
\lesssim - \|\Pi_{\nabla C(X(0))^\perp}\nabla \mathcal{E}(Y(t))\|^2
\lesssim
- \big(\mathcal{E}(Y(t) -\mathcal{E}(X^\sharp)\big),
$$
where $\Pi_{\nabla C(X(0))^\perp}$ denotes orthogonal projection onto
$\nabla C(X(0))^\perp$, and $X^\sharp$ denotes the orthogonal projection of
$Y(0)$ onto the set $C(X)=c_0$.
Here, the final inequality follows by the assumed
variational
stability.
Thus,
$\mathcal{E}(Y(t))-\mathcal{E}(X(0))$ decays time-exponentially so long
as $Y(t)$ remains in $B(X(0), \eps)$,
as does therefore $\|\Pi_{\nabla C(X(0))^\perp} \nabla \mathcal{E} (Y(t))\|$.
Thus,
$$
\frac {d}{dt}\langle (\partial X/\partial s)|_{s=0,c^*}, Y(t)-Y(0)\rangle
\le K\|M(Y(t))\nabla \mathcal{E}(Y(t))\|\le Ke^{-\eta t},\quad\mbox{for some}\quad
K,\eta>0,$$
and so
$\langle (\partial X/\partial s)|_{s=0,c^*}, Y(t)-Y(0)\rangle$
converges time-exponentially.
which, together with the
already established convergence to the manifold of equilibrium solutions,
gives stability and phase-asymptotic orbital stability,
with convergence at exponential rate.
\end{proof}

\subsubsection{Asymptotic orbital stability: infinite-dimensional case}\label{s:inf:asymptotic}
Under suitable assumptions guaranteeing compactness of the flow of
\eqref{gengrad}, the Lyapunov argument of Proposition
\ref{lyap_prop} may be generalized to the infinite-dimensional case,
yielding orbital stability with no rate; see \cite{Wr}. This applies
in particular to all of the examples given below, which are (or are
equivalent to) systems of parabolic type on bounded domains.
However, note that this
requires a spectral gap of $A_{\omega^*}$, under which we would expect (assuming
a correspondingly nice $M$, as also holds for our examples), rather,
exponential stability, whereas the exponential stability result of
Proposition \ref{gradstab}, first, does not generalize in simple
fashion to infinite dimensions, due to the presence of multiple
norms, and, second, requires the restrictive assumption \eqref{d1'},
which does not apply the important example of viscoelasticity,
below.

Moreover, there are interesting situations in which the essential
spectrum of $A_{\omega^*}$ and $d\mathcal{F}$, extend to the
imaginary axis; for example, stability of solitary waves in
viscoelastic, Cahn--Hilliard, or chemotaxis systems, as studied
respectively in \cite{Z3,HoZ}, \cite{Ho1,Ho3}, and \cite{PS2}. In
such situations, it has proven useful to separate the questions of
spectral and nonlinear stability; see \cite{Z1} for a general
discussion of this strategy. In particular, for each of the
above-mentioned settings, it can be shown that to show nonlinear
orbital stability, with time-algebraic rates of decay, it is enough
to establish transversality plus strict stability of point spectrum,
whereupon essential spectrum and nonlinear stability may be shown by
separate analysis.\footnote{
In the periodic parabolic examples studied below,
nonlinear stability follows easily from spectral stability
by general analytic semigroup theory \cite{He,Sat}.
}

For both of these reasons, we restrict ourselves in the infinite-dimensional
setting to the simpler treatment of stability of point spectrum
of $d\mathcal{F}$, from which the relevant nonlinear stability result
may in most settings (in particular, all of those mentioned in
this paper) may be readily deduced.

\begin{definition}\label{spec_stabdef}
A stationary solution $X(t)\equiv X(0)\in \mathbf{H}_1$ of \eqref{gengrad}
has orbitally stable point spectrum if
(i) the dimension of the zero eigenspace of
$d\mathcal{F}(X(0))=M(X(0))A_{\omega^*}$ is equal to the dimension
$\dim \ker A_{\omega^*}+\dim c$
of the manifold \eqref{minfam}
of stationary solutions
(transversality),
and (ii) the nonzero point spectrum of $d\mathcal{F}(X(0))$ has strictly
negative real part.
\end{definition}

\begin{definition}\label{lin_stabdef}
A stationary solution $X(t)\equiv X(0)\in \mathbf{H}_1$ of \eqref{gengrad}
is $\mathbf{H}_1\to \mathbf{H}$-linearly asymptotically orbitally stable if, for
$Y(0)\in H_1$, the solution $Y(t)$ of $(d/dt)Y=M(X(0))A_{\omega^*}Y$
converges in $\mathbf{H}$ to
the tangent manifold of the manifold \eqref{minfam}
of stationary solutions (critical points of \eqref{min}).
\end{definition}

Define now the following local linear analog of the
nonlinear condition \eqref{d1}:
\be\label{e1}\tag{E1}
\hbox{\rm
The center subspace of $M(X(0))A_{\omega^*}$ consists of $\ker M(X(0))A_{\omega^*}$;
}
\ee
that is, $0$ is the only pure imaginary eigenvalue of $M(X(0))A_{\omega^*}$,
and associated eigenvectors are genuine.

\bl\label{e1ver}
Sufficient conditions for \eqref{e1} are \eqref{d1'}
(e.g., $M=M^*$)
and $A_{\omega^*}|_{\nabla C(X(0))^\perp}\ge 0$.
\eqref{d1'} by itself implies that there are no
nonzero imaginary eigenvalues of $M(X(0))A_{\omega_*}$.
\el

\begin{proof}
Let $M(X(0))A_{\omega^*}Y=i\mu Y$, $\mu\in \RM$. Then, dropping
$\omega^*$ and $X(0)$, $\langle AY,MA Y\rangle= i\mu \langle
Y,AY\rangle$, hence
$$
\langle AY, (M+M^*)A Y\rangle=2\mathrm{Re} \langle AY,MA Y\rangle= 0,
$$
and, by semidefiniteness of $(M+M^*)$, we conclude that $(M+M^*)AY=0$.
By \eqref{d1'}, this implies $MAY=0$.
This shows that $0$ is the only pure imaginary eigenvalue of $MA$.

Now, suppose that $(MA)^2 Y=0$. Then, by \eqref{d1'},
$M^*A MAY=0$, hence $\langle MAY,A (MAY)\rangle=0$.
As $MAY\in \nabla C(X(0))^\perp$, by \eqref{M_class}, and
$A|_{\nabla C(X(0))^\perp}\ge 0$ by assumption, this implies
that $MAY\in \ker A$,
hence $\langle AY, (M+M^*)AY\rangle=0$ and so $AY\in \ker(M+M^*)=\ker M$
and $MAY=0$ as before.
This shows that there are no generalized eigenvectors in the
zero eigenspace of $MA$, completing the proof.
\end{proof}

\br\label{transrmk} Transversality condition (i) of Definition
\ref{spec_stabdef} concerns well-behavedness of the existence
problem. Under \eqref{e1}, it is guaranteed by our hypotheses
\eqref{h1}--\eqref{h4}, through \eqref{orthrel}, and the associated
relation $\ker M(X(0))A_{\omega_*}= \ker A_{\omega_*}\oplus \Span
\{\partial X/\partial \omega_j)(\omega_*):j=1,\dots,r\}$. \er

Along with \eqref{e1}, we identify the following mild
properties of the linearized solution theory:\footnote{
This is satisfied for all of the examples of this paper, not
only on the unstable subspace, but all of $\HM$ as expected.}
\be\label{e2}\tag{E2}
\hbox{\rm
If $X(0)$ is asymptotically orbitally stable,
then solutions of $(d/dt)Y=M(X(0))A_{\omega^*}Y$}
\ee
with $Y(0)$
in the constrained unstable subspace of $A_{\omega_*}$
exist for all time and
converge in $\mathbf{H}$ to the

\medskip
\noindent center subspace of $M(X(0))A_{\omega^*}$.

\be\label{e3}\tag{E3}
\hbox{\rm
$A_{\omega_*}$ is bounded on $\HM$ as a quadratic form.}
\ee

\bt\label{specthm}
Assuming \eqref{h1}, \eqref{h2'},\eqref{h3}--\eqref{h4}, \eqref{b1}-\eqref{b2'},
and \eqref{e1},
a stationary solution $X(t)\equiv X(0)\in \mathbf{H}$ of \eqref{gengrad}
has orbitally stable point spectrum if and is $\mathbf{H}\to\mathbf{H}$ linearly orbitally
asymptotically stable only if
$X(0)$ is a strict local minimizer of \eqref{min},
or, equivalently,  the signature of $(\partial c/\partial \omega)(\omega^*)$
satisfies condition \eqref{sorbvar}.
Assuming also \eqref{e2}--\eqref{e3},
$X(0)$ has strictly orbitally stable
point spectrum if and only if $X(0)$ is strictly orbitally stable, or
\eqref{sorbvar}.
\et

\begin{proof}
($\Leftarrow$) (i) Dropping $\omega^*$ and $X(0)$, $MAY=0$ is
equivalent to $AY=\sum_{j=1}^{r}\alpha_j\nabla C_j(X(0))$, for some
$\alpha\in \C^r$. Since $A\frac{\partial X}{\partial
\omega_j}(\omega^*)=-\nabla C_j(X(0))$ for all $j=1,\dots,r$, we
thus have $$A\Bigl(Y+\sum_{j=1}^{r}\alpha_j\frac{\partial
X}{\partial \omega_j}(\omega^*)\Bigr)=0,$$ or $Y\in \Span \{
\partial X/\partial \omega_j(\omega^*):j=1,\dots,r\}\oplus \Span\{\nabla C_j(X(0)):j=1,\dots,r\}$ as claimed.

(ii) By \eqref{e1}, it is sufficient to show that $\mathrm{Re} \lambda\le 0$
whenever $-M(X(0))A_{\omega^*}Y=\lambda Y$.
Dropping $\omega^*$ and $X(0))$ again, we find
from $-\mathrm{Re} \langle AY,MAY\rangle= \mathrm{Re} \lambda\, \langle Y,AY\rangle$
that either $AY=0$, hence $\lambda=0$, or
$
\mathrm{Re} \lambda =\frac{-\langle AY,MAY\rangle}{\langle Y,AY\rangle}
\le 0.
$

($\Rightarrow$)
On the other hand, suppose $X(0)$ is not strictly orbitally stable.
Then, either $A_{\omega^*}|_{\nabla C(X(0))^\perp}$ has a zero eigenvector
not in
the tangent subspace $(\partial X/\partial s)$ of the family
of equilibria described in \eqref{minfam}, violating transversality,
or else
$\check { \mathcal{E}}(Y(0)):=\langle Y(0), A_{\omega^*} Y(0)\rangle <0$
for some $Y(0)\in \nabla C(X(0))^\perp$,
in which case we see from
$(d/dt) \langle Y(t), A_{\omega^*}Y(t)\rangle \le 0$
that $\check{ \mathcal{E}}(Y(t))<0$
for all $t\ge 0$
for the solution $Y$ of the linearized equations with initial data $Y(0)$.
Thus, $Y(t)$ cannot approach the zero-$\check{\mathcal{E}}$
set tangent to the manifold of nearby minimizers as
$t\to \infty$, violating linear asymptotic orbital stability.
Assuming also \eqref{e2},
so that we have transversality by Remark \ref{transrmk},
we find that there are strictly unstable
spectra of $M(X(0))A_{\omega_*}$, or else \eqref{e1}--\eqref{e3}
would yield convergence to the manifold of nearby
minimizers, a contradiction.
\end{proof}

\br\label{grillakisrmk} It would be interesting to establish a
relation between the number of constrained unstable eigenvalues of
$A_{\omega^*}$ and the number of unstable eigenvalues
$n_-(M(X(0))A_{\omega^*})$ of the operator $-M(X(0))A_{\omega^*}$,
similar to those obtained for $A_{\omega^*}$ and
$J(X(0))A_{\omega^*}$ in the Hamiltonian case (see, e.g.,
\cite{GSS,BJK}). These are equal (under appropriate assumptions
guaranteeing that eigenvectors remain in appropriate spaces under
the action of $M(X(0))^{\pm 1/2}$) when $M(X(0))$ is symmetric
positive definite, by the similarity transformation
$M(X(0))A_{\omega^*}\to
M(X(0))^{-1/2}M(X(0))A_{\omega^*}M(X(0))^{1/2}=
M(X(0))^{1/2}A_{\omega^*}M(X(0))^{1/2}$. More generally, they are
equal if $M(X(0))+M(X(0))^*>0$, by homotopy to the symmetric case,
together with the observation that, by \eqref{d1'} and Lemma
\ref{e1ver}, zero is the only possible imaginary eigenvalue of
$M(X(0))A_{\omega^*}$, and, since
$M(X(0))A_{\omega^*}M(X(0))A_{\omega^*}Y=0$ implies
$A_{\omega^*}M(X(0))A_{\omega^*}Y=0$ and thus $\mathrm{Re} \langle
A_{\omega^*}Y,M(X(0))A_{\omega^*}Y\rangle=0$, hence $Y\in \ker
A_{\omega^*}$, from which observations we may deduce that no
eigenvalues cross the imaginary axis throughout the homotopy.

We conjecture that, more generally, under assumptions \eqref{d1'}
and \eqref{h1},\eqref{h2'},\eqref{h3}, \be\label{nconj}
n_-(M(X(0))A_{\omega^*})=\sigma_-(A_{\omega^*}|_{\nabla C^\perp(X(0))})=
\sigma_-(A_{\omega^*})-\sigma_-(dc/d\omega)(\omega_*)), \ee where \eqref{h3} is
needed only for the second equality. In the finite-dimensional,
diagonalizable, case, this follows easily from the observation that
$M(X(0))A_{\omega^*}Y=\mu Y$, $\mathrm{Re} \mu\gtrless 0$ implies that $Y\in \nabla
C(X(0))^\perp$, $Y\not \in \ker A_{\omega^*}$, and $\langle A_{\omega^*}Y,
(M(X(0))+M(X(0))^*)A_{\omega^*}Y\rangle> 0$, and, moreover, $\mathrm{Re} \langle
A_{\omega^*}Y,M(X(0))A_{\omega^*}Y\rangle=\mathrm{Re} \mu \langle Y,A_{\omega^*}Y\rangle$, so that $
\langle Y,A_{\omega^*}Y\rangle \gtrless 0$, whence $n_-(M(X(0))A_{\omega^*})\leq
\sigma_-\Bigl(A_{\omega^*}|_{\nabla C(X(0))^\perp}\Bigr)$. and $n_+(M(X(0))A_{\omega^*})\leq
\sigma_+\Bigl(A_{\omega^*}|_{\nabla C(X(0))^\perp}\Bigr)$.
Recalling that \eqref{h1},\eqref{h2'}
(via \eqref{perprel}) imply
$n_0(M(X(0))A_{\omega^*})=\dim \ker M(X(0))A_{\omega^*}= \sigma_0(A_{\omega^*}) + \dim \ker M(X(0))$, we thus obtain the
result.  Galerkin approximation (with well-chosen subspaces including
$\ker (M(X(0))A_{\omega^*})$) then yields the result in the case that $M(X(0))A_{\omega^*}$ has
a spectral gap and finitely many stable eigenvalues.
\er

\section{Examples: applications to co-periodic stability of periodic waves}\label{s:per}

\subsection{Cahn--Hilliard systems}\label{s:CH}
First, we consider the Cahn--Hilliard system \be\label{CH} u_t=\partial_x^2
(\nabla W(u)-\partial_x^2 u)
\ee where $u\in\RR^m$ and $W:\RM^m\to\RM$ is a smooth function. This
model can be expressed in the form \eqref{gengrad}
with \be\label{CHE}
\mathcal{E}(u):=\int_0^1\Bigl(\frac12 |u_x|^2 + W(u)\Bigr)dx \quad\mbox{and}\quad M:=-\partial_x^2\geq 0\ee
\be\label{CHM}  \nabla_{L^2}
\mathcal{E}(u)= \nabla W(u)-\partial_x^2 u. \ee
Note that scaling $x,t,$ and $W$, we can restrict to the unit interval without loss of generalization. We note that the energy is continuously differentiable on $H^1_\mathrm{per}[0,1]$ and $L^2$-gradients are continuous on $H^2_\mathrm{per}[0,1]$. The system fits our framework so that we can derive stability and instability results that recover and substantially extend the spectral stability
conclusions of \cite{Ho1} for spatially periodic solutions of
\eqref{CH}.  We define the constraint function by
$$C(u)=\int_0^1 u(x)dx.$$
To find solutions of \eqref{critlag} for the Cahn-Hilliard case, we
need to solve the nonlinear oscillator problem
\begin{equation}\label{SCH-CH}
-u_{xx}+\nabla W(u)=-\omega^\ast,
\end{equation}
with $\omega^\ast\in\RR^m$. Whenever the potential $W$ is not convex, \eqref{SCH-CH} possesses families of periodic
solutions which are in fact critical points of \eqref{CH}. For the linearization at such periodic solutions, we find simply the second derivative of the energy since the constraint
function  $C$ is linear:
\begin{equation}\label{A-omega-CH}
A_{\omega^\ast}=\nabla^2_{L^2}\mathcal{E}(u^\ast)-\omega^\ast\nabla^2_{L^2}C(u^\ast)=-\partial_x^2+\nabla^2W(u^\ast).
\end{equation}
In the scalar case, $m=1$, consider periodic patterns with minimal period $1$, first. These possess precisely 2 sign changes, so that by Sturm-Liouville theory, the translational eigenfunction $u^\ast_x$ is either the second or third eigenvalue of  $A_{\omega^\ast}$. More precisely, one can decompose the spatially periodic linearized problem on $[0,1]$ into a problem on even and odd functions, which, equivalently, satisfy the linearized problem on $[0,1/2]$ with Neumann and Dirichlet conditions, respectively. Since $u^\ast_x$ has a sign on $[0,1/2]$, the linearization is semi-definite on the Dirichlet subspace. The Neumann problem can have $\sigma_-=1$ or $\sigma_-=2$. It is well-known that the index in the Neumann case depends on the period-amplitude relation. In fact, periodic orbits patterns in the scalar case come in families parameterized by the maximum $u_\mathrm{max}$ of the solution, once we allow the spatial period $L$ to vary. Whenever the period is increasing with
 amplitude, we find $\sigma=-1$,
decreasing period corresponds to $\sigma=-2$.
We refer the reader to \cite{Korman,Schaaf} and references therein for proofs
in the analogous Dirichlet case.
Alternatively, this could be (but to our knowledge has not been)
deduced by a co-periodic stability index computation like those
in \cite{OZ1,Ho1} in the conservation law case,
which would yield that the Morse index be even or odd accordingly
as $du_\mathrm{max}/dL$ (related by a nonvanishing
factor to the derivative of the Evans function at $\lambda=0$)
is less than or greater than zero.
In particular, we find that periodic patterns are co-periodically
stable if and only if $du_\mathrm{max}/dL>0$ \emph{and}
$dc/d\omega>0$.

In the system case, we obtain a partial analog,
namely, a necessary and sufficient condition for
co-periodic stability of periodic solutions of \eqref{CH} is
that the number of positive eigenvalues of $(\partial c/\partial \omega)(\omega^*)$ is equal to the number of
negative eigenvalues of
$A_{\omega^*}=-\partial_x^2 +\nabla^2 W(u^\ast)$, considered
as an operator on all of $L^2[0,1]_{\mathrm per}$
(i.e., with no zero-mass restriction).
{\it This result appears to be new,}
despite the considerable attention to Cahn--Hilliard equations in the
pattern-formation and phase-transition literature.

Moreover, for the important case of multiply-periodic solutions in
$\RM^m$, which appear (see \cite{Ho2}) not to have been treated so
far either in the scalar or system case, we obtain corresponding
stability conditions by exactly the same argument. Specifically, we find that a necessary and sufficient condition
for co-periodic stability is that the number of positive eigenvalues of $(\partial
c/\partial \omega)(\omega^*)$ is equal to the number of negative
eigenvalues of eigenvalues as does $A_{\omega^*}= -\partial_x^2
+\nabla^2 W(u^\ast)$, considered as an operator on
$L^2[0,1]_{\mathrm per}$. {\it These results too appear to be new}.

Finally, we collect the results  of this subsection in
the following proposition.
\begin{proposition}\label{stab-CH}
Assume that $u^\ast$ is a \textbf{periodic or multi--periodic}
pattern of \eqref{CH}.  The following assertions hold true:
\begin{itemize}
\item[(i)] If $m=1$ and $x\in [0,1]_\mathrm{per}$, the pattern $u^\ast$ is stable if and only if $(dc/d\omega)(\omega^\ast)>0$ and $dL/du_\mathrm{max}>0$;
\item[(ii)] If $m>1$ or for multi-periodic patterns, $u^\ast$ is stable if and only if
$-(\partial c/\partial \omega)(\omega^\ast)$ has the same number of
negative eigenvalues as does $A_{\omega^*}= -\partial_x^2 +\nabla^2 W(u^\ast)$,
considered as an operator on $L^2[0,1]_{\mathrm per}$.
\end{itemize}
\end{proposition}

\br\label{eval} Integrating the traveling-wave ODE \eqref{SCH-CH}
over one period, we find further that \be\label{omegacomp} \omega^*=
-\nabla W(u^*)^a, \ee where superscript $a$ denotes average over one
period of the periodic solution $u^*$, here indexed by its mean $c=
\int_0^1 u^*(x) dx$. This gives the more explicit stability
condition \be\label{explicitcondition} d( W'(u^*)^a)/dc <0 \;
\hbox{(case $m=1$)} \ee or $\sigma_- \big( \partial(\nabla
W(u^*)^a)/\partial c\big)=\sigma_- \big( A_{\omega^*}\big) $ (case
$m> 1$). \er

\br\label{otherobs} By \eqref{omegacomp}, $(dc/d\omega)(\omega^*)$
simplifies in the long-wavelength limit toward a homoclinic solution
with endstate $u^\infty$ as $x\to \pm\infty$ simply to
$-W''(u^\infty)$,
which for $m=1$ (by the existence criterion that $u^\infty$ be a
saddle point)
is scalar and negative,
implying co-periodic instability. Stable, large-wavelength patterns can be found in the heteroclinic limit, when the periodic solutions limit on a heteroclinic loop. In the small-amplitude limit, that is, for almost constant periodic solutions, both coperiodic stability and instability can occur. In the simple scalar cubic case, $W'(u)=L^2(-u+u^3)$, small-amplitude patterns exist for $|c|<1/\sqrt{3}$ for an appropriate length parameter $L$. These small amplitude patterns are stable when $|c|<1/\sqrt{5}$ and unstable when
$|c|>1/\sqrt{5}$; see \cite{GN}.
 \er

\subsubsection{Sturm-Liouville for systems}\label{s:SL}
It is a natural question to ask whether one could obtain information
about $\sigma_-(A_{\omega^*})$ also in the system case $m>1$ using
generalizations of Sturm-Liouville theory such as the Maslov index.
As described in \cite{A,E,U}, there exist such generalizations not
only for operators of the form $A_{\omega^*}$, but for essentially
arbitrary ordinary differential operators $A_{\omega^*}$ of
self-adjoint type. However, so far as we know, all of these
generalizations involve information afforded by evolving planes
spanned by $d$ different solutions, rather than the single zero
eigenfunction $u_x^*$ from which we wish to draw conclusions.
Indeed, in the absence of additional special structure (as described
for example in Cor. \ref{contcor}), we do not know how to deduce
such information analytically.

On the other hand, we note that this point of view does suggest an
intriguing numerical approach to this problem. Namely, borrowing the
point of view espoused in \cite{E,U} that Sturm theory amounts to
the twin properties of monotonicity in Morse index of the operator
$A_{\omega^*}-\lambda \Id$ on domain $[0,x]$ with respect to
frequency $\lambda$ and spatial variable $x$, we find that we may
compute the number of eigenvalues $\lambda< \lambda_0$ of
$A_{\omega^*}$ on domain $[0,x_0]$  by counting instead the number
of {\it conjugate points} $x\in (0,x_0)$ of
$A_{\omega^*}-\lambda_0\Id$, i.e., values for which
$A_{\omega^*}-\lambda_0\Id$ has a kernel on $[0,x]$. This in turn
may be computed using the same {\it periodic Evans function} of
Gardner \cite{G}, as described for example in \cite{BJNRZ}. that
would be used to study behavior in $\lambda$. Some apparent
advantages of this approach are that (i) it avoids the
computationally difficult large-$|\lambda|$ regime, and (ii) the
computational expense is no greater in computing values of the Evans
function along all of $[0,x_0]$ than in evaluating at the single
value $x=x_0$, which in any case requires integrating the eigenvalue
ODE over the whole interval. This seems an interesting direction for
further investigation.

\subsection{Coupled conservative--reaction diffusion equations}\label{s:KellerSiegel}
We consider the general model
\begin{equation}\label{Cons-Diff}
\begin{array}{ll} u_t=\pa_x\Big(a(u)u_x-b(u)v_x\Big),\\
v_t=v_{xx}+\delta u+g(v), \end{array}
\end{equation}
where $a$, $b$, $g$ are of class $C^3$ and $a(u)\geq a_0>0$,
$b(u)\geq b_0>0$, $u,v\in \RM$, $\delta > 0$. We note that this
system can be expressed in the form \eqref{gengrad} with
\be\label{Cons-Diff-E} \mathcal{E}(u,v):= \int_0^1\big( F(u) -uv -
\delta^{-1} G(v) + (2\delta)^{-1}|\partial_x v|^2\Big) dx, \ee where
$F''(u):= \frac{a(u)}{b(u)}$ and $G'(v)=g(v)$, and \be\label{KSM}
\qquad M(u,v):=
-\bp \partial_x\left( b(u)\partial_x\right) & 0\\
0 & \delta\ep \ge 0, \qquad \nabla_{L^2} \mathcal{E}(u,v)=
\bp F'(u)-v\\
-\delta^{-1}\partial_x^2v -u -\delta ^{-1} g(v) \ep. \ee
This general form of a conservation law coupled to a scalar reaction-diffusion equation includes a variety of interesting models such as the Keller-Segel equations of chemotaxis and many of their variations, reaction-diffusion systems with stoichiometric conservation laws, super-saturation models for recurrent precipitation and Liesegang patterns, as well as phase-field models for phase separation and super-cooled liquids; see for instance
\cite{KS,KR,PS2,PS3}.

Again, these systems fit into our general framework of quasi-gradient systems and we can relate spectral stability to signatures of symmetric forms. Since the system (\ref{Cons-Diff}) is parabolic, spectral stability will readily imply nonlinear asymptotic stability up to spatial translations. We thereby recover and extend stability computations from \cite{PS1}.

In this case the constraint function is given by
$$C(u,v)=\int_0^1 u(x)dx.$$
Next, we look for solutions of system \eqref{critlag}. We note that
in the case at hand this system is
\begin{equation}\label{2.4-KS}
\begin{array}{ll}
F'(u)-v=-\omega^\ast,\\
-\frac{1}{\delta}\Big(v_{xx}+g(v)\Big)-u=0.
\end{array}
\end{equation}
Since $F''>0$, we can solve  the first equation of
\eqref{2.4-KS} for $u$. We then substitute the solution
$u=\varphi(v-\omega^\ast)$ in the second equation of \eqref{2.4-KS},
and conclude that the $v$--component of any solution of
\eqref{2.4-KS} satisfies the nonlinear Schr\"odinger equation
\begin{equation}\label{SCH-KS}
v_{xx}+\delta\varphi(v-\omega^\ast)+g(v)=0.
\end{equation}
The above $\omega^\ast$-family of equations is a parametrized
nonlinear oscillator, well-known to support, in appropriate
circumstances, a variety of homoclinic, heteroclinic, and
periodic orbits.

As the constraint function $C$ is linear, we find that
\begin{equation}\label{A-KS}
A_{\omega^\ast}=\nabla^2\mathcal{E}(u^\ast,v^\ast)-\omega^\ast\nabla^2C(u^\ast,v^\ast)=\bp \frac{a(u^\ast)}{b(u^\ast)}  & -1\\
-1& -\frac{1}{\delta}\Big(\partial_x^2+g'(v^\ast)\Big)\ep.
\end{equation}
In analogy to the case of the Cahn-Hilliard equation, we can give necessary and sufficient condition
for co-periodic stability of spatially periodic patterns. Denote by $v_\mathrm{max}$ the maximum of $v$ and write $L(v_\mathrm{max})$ for the period of the family of periodic solutions to (\ref{SCH-KS}).
\begin{proposition}\label{KS-per-stab}
The necessary and sufficient
condition for co-periodic stability of spatially periodic
patterns of coupled conservative--reaction diffusion equations \eqref{Cons-Diff}
is $(dc/d\omega)(\omega^*)>0$ \emph{and} $dL/dv_\mathrm{max}>0$.
\end{proposition}
\begin{proof}
 To start, we consider $(u^\ast,v^\ast)$ a
co-periodic solution of \eqref{2.4-KS} of period $1$. We note that
all we need to show is that $\sigma_-(A_{\omega^\ast})=1$.
Since $\varphi=(F')^{-1}$ and $u^\ast=\varphi(v^\ast-\omega^\ast)$,
we obtain that
\begin{equation}\label{phiF}
\varphi'(v^\ast-\omega^\ast)=\frac{1}{F''(\varphi(v^\ast-\omega^\ast))}=\frac{1}{F''(u^\ast)}=\frac{b(u^\ast)}{a(u^\ast)}.
\end{equation}
Since $v^\ast$ is the solution of the nonlinear oscillator
\eqref{SCH-KS} we can again conclude that
$\tilde{A}_{\omega^\ast}:=-\frac{1}{\delta}\Big(\partial_x^2+g'(v^\ast)\Big)$,
\begin{equation}\label{dif-SCH}
\tilde{A}_{\omega^\ast}-\frac{b(u^\ast)}{a(u^\ast)}\quad\mbox{has
only one simple, negative eigenvalue,}
\end{equation}
provided that  $dL/dv_\mathrm{max}>0$.

One can readily check that $A_{\omega^\ast}$ has the following
decomposition:
\begin{equation}\label{dec-A-omega}
A_{\omega^\ast}=S^\ast
T_{\omega^\ast}S,\quad\mbox{where}\quad S=\bp \frac{a(u^\ast)}{b(u^\ast)}  & 0\\
-1& 1\ep\quad\mbox{and}\quad T_{\omega^\ast}=\bp \frac{b(u^\ast)}{a(u^\ast)}  & 0\\
0& \tilde{A}_{\omega^\ast}-\frac{b(u^\ast)}{a(u^\ast)}\ep.
\end{equation}
Since $S_{\omega^\ast}$ is
bounded on $L^2(\RR,\CC^2)$, invertible with bounded inverse, we infer from Proposition \ref{Syl} that
$\sigma_-(A_{\omega^*})=\sigma_-(T_{\omega^*})$.  From \eqref{dif-SCH} and
since the functions $a$ and $b$ are positive, we infer that
there is $\overline{\psi}\in L^2(\RR)\times H^2(\RR)$,
$\overline{\psi}\ne 0$, such that
$T_{\omega^\ast}\overline{\psi}=\lambda\overline{\psi}$ for some
$\lambda<0$ and the restriction of $
T_{\omega^*}|_{\{\overline{\psi}\}^\perp}
\geq 0$. From Proposition~\ref{p2.5}(i) conclude that
$\sigma_-(A_{\omega^*})=\sigma_-(T_{\omega^*})=1$.
\end{proof}

\br
Similarly as in Remark \ref{eval}, we obtain,
integrating traveling-wave equation \eqref{2.4-KS},
\be\label{omegacompcRD}
\omega^*= -(F'(u)-v)^a,
\ee
where superscript $a$ denotes average over one period,
along with relation
\be\label{relcRD}
g(v)^a= -\delta u^a= -\delta c.
\ee
This gives the explicit co-periodic stability condition
$ (d/dc)( F'(u^*)-v^*)^a <0 $.

Again, one can simplify in the homoclinic, long-wavelength limit
with endstate $(u,v)^\infty$ as $x\to \pm\infty$ to
\be\label{homcRD}
(d/d u^\infty) (F'(u^\infty)-v^\infty)<0,
\;
\hbox{\rm  where $v^\infty=-\delta g^{-1}(u^\infty)$}.
\ee
Comparing with the reduced traveling-wave ODE \eqref{SCH-KS}, we
see (cf. \cite{PS2}) that this is the requirement that $(u^\infty,v^\infty)$
be a saddle (center), hence by the associated
existence theory we obtain again co-periodic instability in the homoclinic
limit. In the limit towards constant states, bifurcations can be more complicated and both co-periodic  stability and instability can occur \cite{grin}.
\er

\begin{remark}\label{KS-W}
Wolansky \cite{W} gives a natural extension of \eqref{Cons-Diff} to
chemotaxis systems that are ``not in conflict'' in a certain sense,
with $u\in \RM^m$, $v\in \RM^n$, $m$ and $n$ arbitrary. For this
class of systems, all of our results apply, giving a generalization
of the stability condition to the system case.
\end{remark}

\br
We mention also a more general class of chemotaxis systems identified by
Horstman \cite{H}, possessing a decreasing Lyapunov function
$ \mathcal{E}$, but apparently not coercive in the sense of \eqref{d1'}.
It would be interesting to see whether this class of systems could be
treated similarly as the case of viscoelasticity with
strain-gradient effects, below.
\er

\subsection{Viscoelasticity with strain-gradient effects}\label{s:visco}
The equations of planar viscoelasticity with strain-gradient effects,
written in Lagrangian coordinates, appear \cite{BLeZ,Y,BYZ} as
\ba\label{e:visco}
\tau_t - u_x&=0,\\
u_t - (\nabla W(\tau))_x&= (b(\tau) u_x)_x - (h(\tau_x)\tau_{xx})_{x}. \ea
Here $\tau, u\in \RM^d$, $d=1,2,3$, $b:\RM^d\to\RM$ is a smooth function, $b(\tau)\geq b_0>0$ and $h(\sigma)=\nabla^2\Psi(\sigma)\geq h_0I_d>0$, for some smooth function $\Phi:\RM^d\to\RM$.
Standard choices of $b$ and $h$ are \cite{BLeZ,Y}
\be\label{bpsi}
b(\tau)= \tau_3^{-1} \Id,
\quad
h(\tau_x)=\Id,
\ee
$b$ corresponding to the usual Navier--Stokes stress tensor, and
$h$ to the phenomenological energy $\psi(p)=|p|^2/2$ of the
usual Cahn--Hilliard theory of phase transitions.

Again, we consider this system as a quasi-gradient system with constraints within the framework \eqref{gengrad}, where the energy is given through \be\label{viscenergy}
\mathcal{E}(\tau,u):= \int_0^1 \left( \frac{|u|^2}{2} + W(\tau)
+\Psi(\tau_x) \right)(x) dx
\ee
and
\be\label{viscoM} M(\tau, u):=
-\bp 0 & \partial_x\\
\partial_x & \partial_x (b(\tau)\partial_x)\ep\geq 0.
\qquad \nabla_{L^2} \mathcal{E}(\tau,u)=
\bp \nabla W(\tau)-h(\tau_x) \tau_{xx}\\
u \ep. \ee
The constraints of the system are given by
$$C(\tau,u)=\Big(\int_0^1\tau(x)dx,\int_0^1u(x)dx\Big).$$
Further, if $(\tau^\ast,u^\ast)$ is a solution of \eqref{critlag}, using that the constraint function $C$ is linear, we compute
\begin{equation}\label{A-omega-ve}
A_\omega^\ast=\bp  -\nabla^2\Psi( \tau^*_x)\partial_x^2 +\nabla^2 W(\tau^*)& 0\\
0& 1\ep.
\end{equation}
Looking for solutions of system \eqref{critlag}, in this case  we
note  that the $u$-component should be constant, that is
$u=-\omega_u\in\RM^d$.  The $\tau$-component satisfies the equation
\begin{equation}\label{int-ve}
\nabla W(\tau)-\big(\nabla\Psi(\tau_x)\big)_x+\omega_\tau=0,
\end{equation}
with $\omega_\tau\in\RM^d$. The above equation is a nonlinear oscillator, which under appropriate conditions on $W$ and $\Psi$ (for example, $\Psi$ close to the identity, $W$ not convex, or, more generally, when the matrix pencil $\nabla^2 \Psi(\tau_*) \rho-\nabla^2 W(\tau_*)$ possesses a negative eigenvalue $\rho$) is also
known to have a variety of solutions depending smoothly on
$\omega_\tau$, proving that \eqref{e:visco} has a
$\omega=(\omega_\tau,\omega_u)\in\RM^{2d}$ dependent family of
critical points, constant in the $u$-component.

Note (\cite{Z3,BLeZ,Y}) that
\be\label{diff1}
\frac{d}{dt}\mathcal{E}(\tau,u)=-\int_0^1 \langle u_x(x),b(\tau(x) )u_x(x)\rangle\,dx
\ee
{\it is not coercive}, according to the fact that
$
\frac12(M+M^*)(\tau,u)=
-\bp 0 & 0\\
0& \partial_x (b(\tau)\partial_x)\ep
$
has a larger (indeed, infinite-dimensional) kernel than that of $M$.
However, it is an interesting fact that if
$\frac{d}{dt}\mathcal{E}(\tau,u)=0$, then $u_x\equiv 0$ by \eqref{diff1},
whence $\partial_x^r \tau_t\equiv 0$, all $r\ge 0$, by \eqref{e:visco}(i),
and so
$ \frac{d}{dt}^2\mathcal{E}(\tau,u)=\|u_t\|^2$.
Thus, $\frac{d}{dt}\mathcal{E} \equiv 0$
gives $\frac{d}{dt}(\tau,u)\equiv 0$, so
that constant-energy solutions are stationary as in the coercive
case \eqref{d1'}.
Likewise, we have the corresponding linear property:

\bl\label{Ever}
The condition \eqref{e1} is satisfied at any variationally
stable critical point of \eqref{e:visco}.
\el

\begin{proof}
Denoting $Y=\bp \tau\\u\ep$, assuming $MA_{\omega^*}Y=i\mu Y $ and dropping
$\omega^*$, we find that

$$
MAY=i\mu Y \Rightarrow \mathrm{Re} \langle AY,
\frac12(M+M^*)AY\rangle=0 \Rightarrow \partial_x u=0,
$$
whence $i\mu \tau=u_x=0$, so that $\mu=0$ or else $\tau \equiv 0$
and $u\equiv \const$, giving $MAY=0$ and again $\mu=0$.
Thus, zero is the only pure imaginary eigenvalue of $MA_{\omega^*}$.
Now, suppose that $Y$ is a generalized zero-eigenvector,
$(MA)^2Y=0$, or $AMAY\in \ker M$, i.e., $\partial_x AMAY=0$.
Computing
\be\label{AMA}
\partial_x AMA= \partial_x \bp
0& -(\partial_x^2-\nabla^2W( \tau^*))\partial_x \\
 -\partial_x (\partial_x^2-\nabla^2W(\tau^*))& \partial_x (b(\tau^*)\partial_x) \\
\ep,
\ee
$\partial_x AMA Y=0$ yields, in the $\tau$-coordinate,
$
\partial_x (\partial_x^2-\nabla^2W(\tau^*)\partial_x u=0$,
hence, by the assumed definiteness of $\partial_x^2-\nabla W(\tau^*)$
on  $\ker \Pi_{\nabla C(X^*)^\perp}$, that $\partial_xu=0$. Substituting this into the
$u$-coordinate of \eqref{AMA} then yields
 $-\partial_x^2 (\partial_x^2-\nabla W(\tau^*))\tau=0$,
hence $-\partial_x (\partial_x^2-\nabla W(\tau^*))\tau=0$. Combining
this with $\partial_x u=0$, we obtain $MAY=0$, and so $Y$ is a
genuine eigenfunction. This verifies that all zero-eigenfunctions
are genuine, completing the proof.
\end{proof}

As a corollary, we obtain spectral stability whenever
\eqref{sorbvar} is satisfied, from which stability follows by the
general theory developed in \cite{Y,BYZ,HoZ,Ko,Z1} and references
therein. Notably, this holds both for periodic solutions, on
$H^1[0,1]_{\mathrm per}$ or for front and pulse type solutions on
$H^1(\RM)$. We note that in this case $(\pa c/\pa\omega)$ has a
block-diagonal structure and that $(\pa c_u/\pa \omega_u)=-I_d$.Thus,
$(\pa c/\pa\omega)(\omega^*)$ satisfies \eqref{sorbvar} if and only if $(\pa c_\tau/\pa \omega_\tau)(\omega_\tau^*)$ satisfies \eqref{sorbvar}

In particular, following the discussion of Sturm-Liouville
properties in the Cahn-Halliard example, we conclude that  for
one-dimensional deformations, $\tau,u \in \RM^1$, we find
co-periodic stability of spatially periodic solutions of
\eqref{e:visco} if and only if
$(dc_\tau/d\omega_\tau)(\omega_\tau^*)>0$ and
$dL/d\tau_\mathrm{max}>0$, that is, if period is increasing
monotonically with amplitude. This recovers and extends  the
corresponding results of \cite{OZ1}. The phrasing of the stability
condition in terms of $(dc_\tau/d\omega_\tau)$ yields also further
insight, we feel, that is not evident in the original phrasing of
\cite{OZ1}.\footnote{ In our notation, the necessary condition of
\cite{OZ1} is effectively $\det (\partial \omega/\partial u_-)
\times \det (\partial u_-/\partial c) >0$, with $\omega=f(u_-)$.}

For multidimensional deformations $\tau \in \RM^d$, we obtain as a
necessary and sufficient condition for co-periodic stability that
the number of positive eigenvalues of $(\partial c_\tau/\partial
\omega_\tau)(\omega_\tau^*)$ is equal to the number of negative eigenvalues of $
-\nabla^2\Psi( \tau^*_x)\partial_x^2 +\nabla^2 W(\tau^*)$,
considered as an operator on all of $L^2[0,1]_{\mathrm per}$ (i.e.,
with no zero-mass restriction). {\it This result appears to be
completely new.}

Moreover, though we do not state it in detail here, we obtain a
corresponding characterization of stability also for multiply-periodic solutions of the full $d$-dimensional equations of
viscoelasticity with strain-gradient effects, $d=1,2,3$, for which $\tau\in
\RM^{d\times d}$ and $u\in \RM^d$, yielding $d^2+d$ constraints
$C_j$ on constrained minimization problem \eqref{min}, again
concerning the mechanical energy \be\label{menergy}
\mathcal{E}(\tau,u):= \int_0^1 \Big( \frac12|u|^2 + W(\tau) +
\Psi(\tau_x)\Big)dx. \ee We obtain as a necessary and sufficient
condition for co-periodic stability that the number of positive
eigenvalues of $(\partial c_\tau/\partial \omega_\tau)(\omega_\tau^*)$, a symmetric
$(d^2+d)\times (d^2+d)$-dimensional tensor,  is equal to the number
of negative eigenvalues of $-\nabla^2\Psi( \tau^*_x)\partial_x^2
+\nabla^2 W(\tau^*)$. For discussion of the full equations and their
derivation, see for example \cite{BLeZ,D}. {\it This result, again,
appears to be completely new}; indeed, the approach by other (e.g.,
Evans function) means appears quite forbidding in likely complexity.
Again, we collect the results from this subsection in the following proposition.
\begin{proposition}\label{V-E-stability}
Assume that $(\tau^\ast,u^\ast)$ is a \textbf{1D-periodic or multiD--periodic} pattern of \eqref{e:visco} The following assertions hold true:
\begin{itemize}
\item[(i)] For the case of one-dimensional deformations $\tau\in\RR$, the periodic pattern $(\tau^\ast,u^\ast)$ is stable if and only if
$(dc_\tau/d\omega_\tau)(\omega_\tau^\ast)>0$ and $dL/d\tau_\mathrm{max}>0$ ;
\item[(ii)] For the case of multidimensional deformations $\tau \in \RM^d$, the periodic pattern
$(\tau^\ast,u^\ast)$ is stable if and only if the number of positive eigenvalues of $(\partial c_\tau/\partial
\omega_\tau)(\omega^\ast)$ is equal to the number of negative eigenvalues of $
-\nabla^2\Psi( \tau^*_x)\partial_x^2 +\nabla^2 W(\tau^*)$,
considered as an operator on all of $L^2[0,1]_{\mathrm per}$;
\item[(iii)] For the case of full $d$-dimensional equations of viscoelasticity with strain-gradient effects, the multiply-periodic pattern $(u^\ast,v^\ast)$ is stable if and only if the number of positive
eigenvalues of $(\partial c_\tau/\partial \omega_\tau)$, a symmetric
$(d^2+d)\times (d^2+d)$-dimensional tensor,  is equal to the number
of negative eigenvalues of $-\nabla^2\Psi( \tau^*_x)\partial_x^2
+\nabla^2 W(\tau^*)$ considered as an operator on all of $L^2[0,1]_{\mathrm per}$.
\end{itemize}
\end{proposition}

\br
It has been shown by other means in \cite{BYZ,Z3} that variational
and time-evolutionary stability are equivalent for \eqref{e:visco}; here,
we recover these results as a consequence of our general theory.
We note that it was shown in \cite{BYZ,Z3} also that
the spectrum of $MA_{\omega^*}$ is real for $\mathrm{Re}\lambda \ge 0$,
which is special to the precise structure of the model, and
does not hold in general.
\er

\br\label{visceval}
As in Remark \ref{eval}, we find,
integrating the traveling-wave ODE \eqref{int-ve}
over one period, that
$\omega_\tau^*= -\nabla W(\tau^*)^a$,
where superscript $a$ denotes average over one period,
giving the explicit (singly periodic) stability condition
\be\label{vexplicitcondition}
\hbox{\rm
$d(W'(\tau^* )^a)/dc_\tau <0$
and $dL/d\tau_\mathrm{max}>0$
(case $d=1$)}
\ee
or
$\sigma_- \big( \partial(\nabla W(\tau^*)^a)/\partial c_\tau\big)=\sigma_- \big( A_{\omega^*}\big) $ (case $d\ge 1$), where
$c_\tau=\int_0^1\tau^*(x)dx$.
\er

\br\label{contrmk}
On their natural spaces
$\HM=(H^1\times L^2)[0,1]_{\rm per}$ and $\HM=(L^2\times H^1)[0,1]_{\rm per}$,
respectively (the spaces on which the second variation is bounded and coercive),
\eqref{viscenergy} and \eqref{Cons-Diff-E} are not only not $C^2$,
but not even continuous, since $\tau_x$ and $u$
enter nonlinearly, but are not controlled in $L^\infty$ by the norm
associated with $\HM$.
\er

\subsection{Generalized KdV}\label{s:gKdV}
It is well-known from the celebrated work of Grillakis--Shatah--Strauss (\cite{GSS}),
that the stability of solitary waves arising in models
such as the Klein-Gordon equation or the
generalized KdV equation
can be studied using
variational results similar to our results. We briefly mention
here the generalized KdV model and the corresponding functionals.
Consider the equation
\begin{equation}\label{genKdV}
u_t+(\nabla F(u))_x=u_{xxx}.
\end{equation}
Here $u\in\RM^m$ and $F:\RM^m\to\RM$ is a smooth function. This
equation can have periodic waves with respect to co-periodic
perturbations as shown in \cite{BJK} or solitary waves on the whole line
as discussed in \cite{GSS}. This equation also fits the general
framework \eqref{e:genham} with
\begin{equation}\label{energy-KdV}
\mathcal{E}(u):=\int_0^1\Bigl(\frac12 |u_x|^2 + F(u)\Bigr)dx
\end{equation}
and
\begin{equation}\label{KdV2}
J=\pa_x,\qquad \nabla_{L^2} \mathcal{E}(u)= \nabla F(u)-\partial_x^2
u.
\end{equation}
In the case of periodic patterns we can take $x\in [0,1]_{\rm per}$
and consider \eqref{genKdV} in the energy space
$\HM=H^1_\mathrm{per}[0,2]$, with mass constraint
$C(u)=(\int_0^1 u(x)dx,\int_0^1 u^2(x)dx).$

\br\label{diffrmk}
It is readily verified that energies \eqref{CHE},\eqref{energy-KdV}
are $C^2$ on $\HM=H^1[0,1]_{\rm per}$ or $H^1(\RM)$.
Moreover, there is a weak existence theory \cite{Me} on the
energy space $\HM$,
verifying condition \eqref{b2}.
\er

\subsection{General conservation laws}\label{ozrmk}
For periodic solutions of general parabolic conservation laws
$$ u_t +(f(u))_x=(B(u)u_x)_x, $$
{\it without the assumption of generalized
gradient form}, Oh and Zumbrun (\cite{OZ1}, Thm. 5.9) have obtained
by direct Evans function/Stability index techniques the related {\it
necessary stability condition } that $\det\Bigl((\partial
c/\partial\omega)( \omega^*)\Bigr)$ have a certain sign $\sgn
\gamma$ determined by the orientation properties of $u^*$ as a
solution of the underlying traveling-wave ODE. The Jacobian
$(\partial c/\partial\omega)( \omega^*)$ is connected to the
(necessary) stability index condition of Theorem 5.9, \cite{OZ1}
through the observation that, in the notation of the reference,
$$
\det (\partial u_m^*/\partial u_-)\det df(u_-)=
\det(\partial u_m^*/\partial u_-)(\partial u_-/\partial q)=
\det (\partial u_m^*/\partial q)=\det\Bigl((\partial c/\partial\omega)( \omega^*)\Bigr),
$$
where $q=\omega^*$ and $u_m^*=c$.
This is somewhat analogous
to the relation established in \cite{PW} between Evans and variational
stability conditions in the Hamiltonian solitary wave case; see
also \cite{Z2}.
The sign $\sgn \gamma$ may be recognized, by a homotopy argument
similar to that used to prove the full index relation,
as measuring the parity of the sign of the Morse index of
the linearized traveling-wave
operator $\Bigl(B(u^*(x)\partial_x+dB(u^*(x))u_x^*\Bigr)- df(u^*(x))$
equivalent (by integration) to the restriction to
the set of zero-mass functions $\nabla C(u^*)^\perp$
of the full linearized operator
$$
\partial_x\Bigl(B(u^*(x))\partial_x+dB(u^*(x))u_x^* - df(u^*(x)) \Bigr),
$$
where $C(u):=\int_0^1u(x)dx,$ thus completing the analogy to the generalized
gradient case.
Indeed, as this argument reiterates, the stability index itself may be
viewed as a mod two
version of the classical Sturm--Liouville theory applicable to the
non-selfadjoint case; see especially the discussion of \cite{E,U}
from a similar perspective of the general self-adjoint case, for which
the zeros counted in the second-order scalar case are replaced by zeros of
an appropriate Wronskian.

\section{Sideband instability and modulational dichotomies}\label{s:sideband}
We now describe some
interesting consequences regarding the Whitham modulation
equations and modulational stability of periodic waves,
in particular, the modulational dichotomies mentioned in the introduction.
The Whitham modulation equations, a formal WKB expansion in the
high-frequency limit, is known (see \cite{Se,OZ2,JZB,BJNRZ,FST} and especially
\cite{NR1,NR2}) to have a dispersion relation agreeing to lowest order
with the spectral expansion of the critical eigenmodes of $L_{\omega^*}$
determining low-frequency (sideband) stability.
In particular, {\it well-posedness of the Whitham equation} may be seen
to be {\it necessary for low-frequency modulational stability.}

As critical, or slow, modes are associated to linear order with
variations $\partial X/\partial \omega$ and $\partial X/\partial s$
along the manifold of nearby stationary solutions, the equations naturally
involve the same quantities already studied in relation to co-periodicity,
and this turns out to yield in some situations
substantial
further information.
For details of the derivation of the Whitham equations
in the generality considered here, see, for example, \cite{Se,OZ2}.

\subsection{Viscoelasticity with strain-gradient effects}
The first-order Witham system for \eqref{e:visco} is
(see \cite{JNRZ2}):
\ba\label{whitham}
\tau^a_t - u^a_x&=0,\\
u^a_t - \nabla W(\tau)^a_x&=0,\\
k_t&=0,
\ea
where superscript $a$ denotes average over one period,
$k$ denotes wave number, or one over period.
By the block-triangular structure, we may ignore the trivial
third equation and concentrate on the first two.
Indeed, this is a common feature of all of our dissipative examples,
as the general $k$-equation is $k_t-(\sigma(\omega, k)k)_x=0$,
where $\sigma$ denotes speed of the associated periodic traveling wave,
which for gradient-type systems is identically zero.


Integrating the traveling-wave ODE \eqref{critlag}
$$
-(\nabla\Psi(\tau_x))_x + \nabla W(\tau) + \omega=0,
$$
over one period, where $\omega$ is the Lagrange
multiplier for the associated constrained minimization problem
in $\tau$, we find that $\nabla W(\tau)^a=-\omega$,
hence \eqref{whitham} writes more simply as
\ba\label{whitham2}
c_t - d_x&=0,\\
d_t + \omega_x&=0,
\ea
where $c=\tau^a $ is the constraint for the minimization problem,
$d=u^*$ is a parameter for the minimization problem (irrelevant,
by Galillean invariance with respect to $u$),
and $\omega=\omega(c)$ independent of $u$ is determined by the solution
structure of the standing-wave problem.
This is of exactly the same form as the equations of viscoelasticity
themselves, whence the condition for hyperbolicity of the equations
is that $\partial \omega/\partial c$ (recall, this is a Hessian for the constrained
minimization problem, so symmetric) be {\it negative definite}.

It has been shown that hyperbolicity of the Whitham system is
necessary for sideband stability.
This means that periodic waves are sideband
stable {\it only if}
$\partial \omega/\partial c<0$, in which case
Corollary \ref{stabcor}, implies that they are co-periodic stable
only if the second variation operator
$$
-\nabla^2\Psi( \tau^*_x)\partial_x^2 +\nabla^2 W(\tau^*)
$$
is nonnegative on all of $H^1[0,1]_{\rm per}$,
not just for perturbations with zero mean.
This is false in the scalar case $\tau\in \RM$, by Sturm-Liouville
theory and the fact that $\tau_x^\ast$ is a zero-eigenfunction that
changes sign. Thus, we recover in straightforward fashion the
dichotomy result of \cite{OZ1}.
\begin{theorem}\label{V-E-dich} In the scalar case, $\tau,u\in\R$, assume that
 $(\tau^\ast,u^\ast)$ is a periodic solution of \eqref{e:visco}. Then $(\tau^\ast,u^\ast)$ is either co-periodic unstable
or else it is sideband unstable: in all cases, modulationally unstable.
For systems $\tau,u\in\R^d$,
solutions are modulationally unstable
unless $A_{\omega^*}=-\nabla^2\Psi( \tau^*_x)\partial_x^2 +\nabla^2 W(\tau^*)$
is nonnegative on all of $H^1[0,1]_{\mathrm per}^d$.
\end{theorem}
{\it This yields for the first time an extension to multidimensional deformations
of the dichotomy of \cite{OZ1}.}
Moreover, it gives new information even in the one-dimensional
deformation case,
removing a superfluous technical assumption (see Theorem
7.1, \cite{OZ1}) that period $T$ be increasing with
respect to amplitude $a$.
This both simplifies the analysis, eliminating the often-difficult
problem of determining $\sgn(dT/da)$,
and answers the question left open in \cite{OZ1}
whether there might exist modulationally stable periodic waves
with period decreasing with respect to amplitude.

Further, though, as discussed in Section \ref{s:SL},
 we cannot say much in the general system case
regarding nonnegativity of $A_{\omega^*}$, due to additional
structure we can in the present situation say substantially more.
Specifically, consider the class of multidimensional deformation
solutions obtained by continuation from the special class of {\it
decoupled solutions}, equal to a one-dimensional deformation
solution in one of its components, say, $(\tau_{j_*},u_{j_*})$, and
constant in the other components $(\tau_j, u_j)$, $j\ne j_*$.
This
generates a large subclass of the possible multidimensional
deformation solutions: in some cases, perhaps, all.

\bc\label{contcor}
In the (singly periodic) system case,
all periodic solutions of \eqref{e:visco}
lying on a branch of transverse
solutions extending from a decoupled solution are
modulationally unstable.
\ec

\begin{proof} By Sturm-Liouville considerations, decoupled solutions
(reducing to the scalar case), satisfy $\sigma_-(A_{\omega^*})\geq 1$,
hence are modulationally unstable. Noting that
$\kernel(A_{\omega^*})$ has constant dimension so long as the
solution remains transversal, we find that no eigenvalues can pass
through the origin and so $\sigma_-(A_{\omega^*})\equiv 1$ along the
entire branch, yielding the result by Theorem \ref{V-E-dich}.
\end{proof}

\br\label{numrmk}
Though transversality is difficult to check analytically, it is
straightforward to implement as part of a numerical continuation study.
From a practical point of view, the main importance of Corollary
\ref{contcor} is that waves obtained through numerical continuation
from decoupled solutions are unstable at least up to the point that
bifurcation first occurs.
\er

\br\label{branchrmk}
It is worth noting that the solutions bifurcating from the
one-dimensional case include truly multi-dimensional ones and
not only one-dimensional solutions in disguise.
As noted implicitly in the description \eqref{whitham},
transverse planar periodic solutions of \eqref{e:visco}
may be locally parametrized, up to $x$-translates,
by the wave number $k=1/L$, where $L$ is the period,
the mean $c\in \RM^d$ of $\tau$ over one period,
and the material velocity $u\equiv \const$-
the latter entering trivially,
by invariance of \eqref{e:visco} under the Galillean transformation
$u\to u+c$, $c$ constant.
Among these solutions are ``disguised'' one-dimensional solutions,
for which $\tau$ is confined to an affine set; see \cite{Y}
for further classification/discussion.
However, a convenient asymmetry in the coefficient
$1/\tau_3$ of the parabolic term in $u$ for the standard
choice of viscosity \eqref{bpsi} forces such solutions to
lie either in the hyperplane $\tau_3\equiv 1$ or else along
the line $\tau_1=\cdots=\tau_{d-1}=0$,
so that solutions can be effectively one-dimensional in this
sense only if $\tau_3\equiv 1$,
or else $\tau_1\equiv \tau_2\equiv 0$.  But,
these possibilities can be excluded by choosing a $\tau$-mean $c$
for which $c_3\ne 1$ and
$(c_1,\dots, c_{d-1})\ne (0,\dots,0)$, of which there
are uncountably many.
\er

\br\label{betteroz}
Theorem \ref{V-E-dich} resolves a question left open in
\cite{OZ1} whether there might exist modulationally stable waves
in the case $dL/d\tau_\mathrm{max}< 0$.
The reason for the difference in these two sets of results
is that the treatment of co-periodic stability in \cite{OZ1} proceeds by
a stability index (Evans function)
computation, which gives information on the parity
of the number of unstable eigenvalues of the linearized operator about
the wave and not the actual number as determined here.
\er

\subsubsection{The multiply periodic case}\label{s:multivisc}
It is interesting to consider the implications of sideband
stability for multiply periodic solutions of the full $d$-dimensional
equations of viscoelasticity with strain-gradient effects, for which the
first-order Whitham
homogenized system is (ignoring the decoupled $k$-equation)
\ba\label{mpwhitham2}
c_t - \nabla_x  d &=0,\\
d_t + \nabla_x  \omega&=0,
\ea
$x\in \RM^d$, with $c$ now a matrix $\in \RM^{d\times d}$ and $d\in \RM^d$,
or $c_{tt}-\nabla_x(\nabla_x \omega(c))=0$.
For this system,
well-posedness (hyperbolicity) is well-known \cite{D} to be equivalent to
{\it rank-one convexity} of $(\partial c/\partial \omega)$,
a less-restrictive condition than the convexity condition
arising in the one-dimensional case $x\in \RM^1$.
(Recall, under this notation, $c=\tau^a\in \RM^{d\times d}$.)
This enforces only $d$ positive eigenvalues in the signature
of $(\partial c/\partial \omega)$, with another $d(d-1)$
zero eigenvalues forced by rotational symmetry,
leaving $d^2-d$ possible negative eigenvalues, in principle compatible
with variational stability.
At the same time, the relation between variational and time-periodic
stability is unclear, as the argument of Lemma \ref{e1ver}
yielding \eqref{e1} breaks down
in the fully multidimensional case.\footnote{Specifically,
$(M+M^*)$ is no longer definite in the $u$ coordinate,
a key element in the proof.}
Thus, it seems possible that one might find {\it stable} multiply periodic
solutions of the equations of viscoelasticity.

\subsection{Cahn--Hilliard systems}
The second-order Whitham modulated system for \eqref{CH} may be seen by a
similar (straightforward) derivation as in \cite{Se} to be
\ba\label{whitham_ch}
u^a_t &= -\partial_x^2( \partial_x^2(u) -\nabla W(u))^a,\\
k_t&=\partial_x(\alpha(u^a,k)\partial_x(u^a)+
\beta(u^a,k)\partial_x k).
\ea
where
$\alpha$, $\beta$, since irrelevant for our considerations here,
are left unspecified.
Integrating the traveling-wave ODE \eqref{critlag}
$$
-u'' + \nabla W(u) + \omega=0,
$$
over one period, where $\omega$ is the Lagrange
multiplier for the associated constrained minimization problem,
we find that $(\nabla W(u)-u'')^a=\nabla W(u)^a=-\omega$,
hence \eqref{whitham_ch} writes more simply as
\ba\label{whitham2ch}
c_t &= -\partial_x^2 (\omega)=
-\partial_x ((\partial \omega/\partial c)\partial_x c ),\\
k_t&=\partial_x(\alpha(c,k)\partial_x c+
\beta(c,k)\partial_x k).
\ea
where $c=u^a$ is the constraint for the minimization problem.
This system is well-posed (parabolic), corresponding to sideband stability, only if
$\partial \omega/\partial c$ is negative definite, with the
same conclusions as in the case of viscoelasticity just considered.

In particular, in the scalar case, recovering a result of
\cite{Ho1}, we may immediately conclude existence of a dichotomy as
in \cite{OZ1}, stating that waves are either sideband unstable or
co-periodic unstable; in either case, unstable.
As in the previous
subsection, we obtain also extensions to the system case.

It is worth mentioning that the one-dimensional theory just
described extends essentially unchanged to the multiply periodic
case, i.e., stationary solutions in $H^1[0,1]^m_{\mathrm per}$.
Again, there is a single constraint $c$ equal to the mean over
$[0,1]^m$, and a single Lagrange multiplier $\omega$, with
standing-wave PDE
$$
\Delta u- \nabla W(u) =-\omega.
$$
In the scalar case, the condition for co-periodic stability
is $\partial c/\partial \omega>0$, by our abstract theory.

On the other hand, the Whitham modulation equations are easily found
to be
$$
\begin{aligned}
c_t&= \Delta \omega= -\nabla_x \cdot  ( (\partial \omega/\partial c) \nabla_x c),\\
k_t&=\nabla_x \cdot (\alpha(c,k)\nabla_x c+
\beta(c,k)\nabla_x k),
\end{aligned}
$$
$k\in \RM^m$,
hence are well-posed only if $(\partial \omega/\partial c)\le 0$.
We have therefore that
multiply periodic waves are either co-periodically
unstable or else sideband unstable, a dichotomy similar to that
observed for singly periodic scalar waves,
implying that multiply-periodic stationary solutions of the
scalar Cahn--Hilliard equation are always modulationally unstable
under assumption \eqref{h3}.
{\it So far as we know, this result for multiply-periodic Cahn--Hilliard
waves is new} (though see Remark \ref{arndrmk} below);
the connection to sideband stability likewise appears to be a novel
addition to the Cahn-Hilliard literature.

In the system case, we obtain a partial analog, similarly as for
viscoelasticity, namely, co-periodically stable solutions of
\eqref{CH} are necessarily sideband unstable, unless $A_{\omega^*}=
-\Delta_x+\nabla^2 W(u^\ast)$ is nonnegative on $H^1[0,1]^m_{\mathrm
per}$. We collect the results of this subsection in the following
theorem.

\begin{theorem}\label{V-E-W-dich}
 Assume $u^\ast$ is a periodic solution of \eqref{CH}.
Then $u^\ast$ is either co-periodic unstable or sideband unstable.
In the system case, co-periodically stable waves are sideband
unstable unless $A_{\omega^*}=-\Delta_x+\nabla^2 W(u^\ast)$ is
nonnegative on $L^2[0,1]^m_{\mathrm per}$.
\end{theorem}

\br\label{arndrmk}
In the scalar, one-dimensional case, we may obtain the result of modulational
({\it but not necessarily sideband}) instability more directly, as in \cite{GN},
by the observation that, by Sturm--Liouville considerations, periodic waves
cannot be co-periodically stable on a doubled domain $[0,2T]$, since the
unconstrained problem has at least two unstable modes, only one of which may
be stabilized by the constraint of constant mass.
Recalling that modulational stability is equivalent to co-periodic
stability on multiple periods $(0,NT)$ for arbitrary $N$, we obtain the result.
In the two-dimensional case, we may obtain, similarly,
co-periodic instability
on the multiple domain $(0,NT]\times (0,NT]$, by the Nodal Domain
Theorem of Courant \cite{CH}, using a bit of planar topological reasoning
to conclude that there must either be a nodal domain in each cell that
is disconnected from the boundary, or else a nodal curve which disconnects
two opposing edges of the cell boundary, and thus there must be at
least $2N$ nodal domains in $(0,NT]\times (0,NT]$, and thereby $2N-1$ unstable
modes,
again yielding modulational instability for $N\geq 2$.
This argument is suggestive also in dimensions $m=3$ and higher, but
would require further work to eliminate the possibility that there
are only two nodal domains in $(0,NT]^m$, for example a thickened
$m$-dimensional lattice (no longer disconnecting space in dimensions
$m>3$) and its complement. \er

\br
It would be interesting to consider the consequences of our
theory in the planar multidimensional case considered by Howard in \cite{Ho2},
that is, the implications for transverse instability.
\er

\subsection{Coupled conservative-reaction systems}

For equations \eqref{Cons-Diff}, it is readily seen that the
second-order
Whitham modulated system is
\ba\label{whitham_KS}
u^a_t &= \partial_x\Bigl(b(u)^a\partial_x (F'(u)-v)\Bigr),\\
k_t&=\partial_x\Bigl(\alpha(u^a,k)\partial_x(u^a)+ \beta(u^a,k)\partial_x k\Bigr),
\ea
so that, denoting $c=u^a$, $b(u)^a=\beta(c,k)>0$,
and
referring to the traveling-wave equation \eqref{2.4-KS},
we obtain the equation
$$
\begin{aligned}
c_t&= -(\beta(c,k) \omega_x )_x=
-(\beta(c,k)(\partial \omega/\partial c)c_x)_x,\\
k_t&=(\alpha(c,k) c_x+ \beta(c,k) k_x)_x,
\end{aligned}
$$
which is well-posed (parabolic)
only if $(\partial c/\partial \omega)(\omega_*)<0$,
in contradiction with co-periodic stability.
Thus, we find a dichotomy as in the previous cases.

Likewise, for multiply periodic solutions of
\begin{equation}\label{mCons-Diff}
\begin{array}{ll} u_t=\nabla_x\cdot\Big(a(u)\nabla_x u
-b(u)\nabla_x v\Big),\\
v_t=\Delta_{x} v+\delta u+g(v), \end{array}
\end{equation}
or stationary solutions in $H^1[0,1]^m_{\mathrm per}$, we obtain
$$
\begin{aligned}
c_t&= -\nabla_x \cdot (\beta(c,k)\partial (\omega/\partial c) \nabla_x c  ),\\
k_t&=\nabla_x \cdot
(\alpha(c,k)\nabla_x c+ \beta(c,k)\nabla_x k),
\end{aligned}
$$
$k\in \RM^m$,
where $c$ is equal to the mean over $[0,1]^m$, and $\omega$ is the
corresponding Lagrange multiplier, again yielding a modulational
dichotomy unless $A_{\omega_*}$ is not nonnegative on
$H^1[0,1]^m_{\mathrm per}$, We collect the results of this
subsection in the following theorem.

\begin{theorem}\label{cRD-dich} Assume that
 $(u^\ast,v^\ast)$ is a periodic solution of \eqref{Cons-Diff}.
Then $(u^\ast,v^\ast)$ is either co-periodic unstable or else it is
sideband unstable: in all cases, modulationally unstable. Multiply
periodic solutions of \eqref{mCons-Diff}  are modulationally
unstable unless $A_{\omega^*}$ is nonnegative on $H^1[0,1]_{\mathrm
per}^m$.
\end{theorem}

\subsection{Generalized KdV}\label{s:kdvside}
Results of \cite{BJ1}
for the generalized KdV equation, relating
sideband stability to Jacobian determinants of action variables,
appear to be an instance of a similar phenomenon in the Hamiltonian
case.
This would be a an interesting direction for further study.

\subsection{Dichotomies revisited}\label{s:revisit}
The above examples can be understood in a simpler and more unified
way through the underlying quasi-gradient structure
of the linearized generator $\mathcal{L}=-MA$.
Namely, we have only to recall that {\it sideband instability}
is defined as stability for $\xi\in \R$, $|\xi|<<1$ of the small
eigenvalues of the Floquet operator
\be\label{Lxi}
\mathcal{L}_\xi:=e^{-i\xi x}\mathcal{L}e^{i\xi x}=
M_\xi A_\xi
\ee
bifurcating from zero eigenvalues of $\mathcal{L}_0=\mathcal{L}$,
where the $\xi$ subscript indicates conjugation by the multiplication
operator $e^{i\xi x}$, or
$$
\mathcal{L}_\xi:=e^{-i\xi x}\mathcal{L}e^{i\xi x},
\quad
M_\xi:=e^{-i\xi x}M e^{i\xi x},
\quad
A_\xi:=e^{-i\xi x}A e^{i\xi x}.
$$
{\it Modulational stability}, similarly,
is defined as stability for $\xi\in \R$, of $M_\xi A_\xi$.

Note first that $A_\xi$ inherits automatically the self-adjoint property
of $A$, and also the quasi-gradient property
\be\label{qprop}
\hbox{\rm $\frac12(M_\xi+M_\xi^*)\ge 0$ for all $|\xi|<<1$,
}
\ee
as seen by the computation
$(M_\xi+M^*_\xi)=
(M+M^*)_\xi=
(M+M^*)^{1/2}_\xi (M+M^*)^{1/2}_\xi$,
both of which follow from
the adjoint rule
\be\label{adjrule}
(N_\xi)^*=(N^*)_\xi,
\ee
where the $*$ on the lefthand side refers to $L^2[0,L]_{\rm per}$
and the $*$ on the righthand side refers to $L^2(\RM)$ adjoint,
valid for any operator $N$ that has an $L^2(\RM)$ adjoint $N^*$,
is $L$-periodic in the sense that
\be\label{Lper}
(Nf)(x+L)= \big(Nf(\cdot+L)\big)(x)
\ee
for $C^\infty$ test functions $f$, and is bounded from $W^{s,\infty}$
to $L^\infty$ for $s$ sufficiently large: in particular, for the
periodic-coefficient differential operators considered here in the
applications.
We note further that, under these assumptions,
$N_\xi$ and $N_\xi^*$ take $L$-periodic functions to $L$-periodic functions.

The latter statement follows by noting that $N_\xi$ is also
invariant under shifts by one period. The former follows by noting
that, for $u$, $v$ periodic, $H^s$  and satisfying \be\label{prop2}
\hbox{ $u(x+L)=\gamma u(x)$, $v(x+L)=\gamma v(x)$, $|\gamma|=1$,}
\ee we have that  $\langle u, Nv\rangle_{L^2[0,L]}$ is equal to  $L$
times the  mean of $u \cdot Nv$  over $\RM$. This follows by smooth
truncation on bigger and bigger domains, with truncation occurring
over a single period, and noting that this bounded error goes to
zero in the limit. But, then, this is also equal to the mean of $N^*
u \cdot v$ over $\RM$, by taking the $L^2$ adjoint for the truncated
approximants and arguing as before by continuity. Finally, observe
that  $\langle e^{-i\xi x}N e^{i\xi x} u, v\rangle = \langle N
\tilde u, \tilde v \rangle$, where $\tilde u=e^{i\xi x}u$ and
$\tilde v=e^{i\xi x}v$ satisfy the above property \eqref{prop2}.
Thus, it is equal to $\langle \tilde u, N^* \tilde v\rangle= \langle
u, (N^*)_\xi v\rangle$ as claimed. For periodic-coefficient
differential operators $p(\partial_x)=\sum_{j=0}^r a_j
\partial_x^j$, \eqref{adjrule}, and the property that periodic functions
are taken to periodic functions, may be verified directly, using
the rule
$$
(p(\partial x))_\xi=p(\partial_x +i\xi).
$$

Now, assume, as is easily verified in each of the cases considered (but
need not always be true), that
\be\label{f1}\tag{F1}
\ker M_\xi=\emptyset,
\ee
so that the corresponding linear evolution equation $(d/dt)Y=-MAY$ is
unconstrained (has no associated conservation laws).
Then, {\it for $M$ self-adjoint}, we obtain the correspondence
\be\label{count}
\sigma_-(A)\leq \sigma_-(A_\xi)=n_-(M_\xi A_\xi),\footnote{
The strict inequality $ \sigma_-(A)<\sigma_-(A_\xi)$ is possible in
the cases we consider, due to bifurcation of $\ker A$.}
\ee
where $n_-$ is the number of negative eigenvalues counted by algebraic
multiplicity,
as may be seen by the similarity transformation
$M_\xi A_\xi\to M_\xi^{-1/2}M_\xi A_\xi M_\xi^{1/2}= M_\xi^{1/2}A_\xi M_\xi^{1/2}$
(Here, we are assuming sufficient regularity in coefficients that eigenfunctions
remain in proper spaces upon application of $M_\xi^{\pm 1/2}$ and also implicitly
that spectrum is discrete, properties again easily verifiable in each of
the cases previously considered).
More generally, the result carries over to the case \eqref{d1'}
provided $(M_\xi+M_\xi)>0$;
see Remark \ref{grillakisrmk}.

Thus, we find immediately that modulational stability is violated
{\it unless $A_\xi$ is unconditionally stable} for $|\xi|<<1$,
which, in the limit $\xi\to 0$, gives unconditional (neutral) stability
of $A$.
Moreover, again invoking continuity of spectra with respect to $\xi$,
we find that, assuming co-periodic stability, or conditional stability
of $A$, that the negative eigenvalues of $M_{\xi}A_\xi$ must
be small eigenvalues bifurcating from zero-eigenvalues of $MA$,
which, by the property \eqref{e1} as guaranteed by Lemma \ref{e1ver},
are the only neutral (i.e., zero real part) eigenvalues of $M A$.

This argument shows that, assuming the property \eqref{count},
first, modulational stability requires unconstrained stability of
$A$, and, second, if co-periodic stability holds, then
sideband stability requires unconstrained stability of
$A$.  The second observation may be recognized as exactly
the modulational dichotomy just proved case-by-case, the first
as a generalization of the one obtained by nodal domain consideration
in Remark \ref{arndrmk}.

\br\label{otherres}
We obtain by this argument modulational dichotomies for Cahn--Hilliard
and coupled conservative--reaction diffusion equations.
On the other hand, we don't (quite) obtain a result for the viscoelastic
case, for which $M$ is not self-adjoint, and so our simple
argument for \eqref{count} does not apply.  We can recover this however
by a ``vanishing viscosity'' argument, noting for $\eps>0$ that
$n_-(M(\eps)_\xi A_\xi)=\sigma_-(A_\eps)$ for $M(\eps)=M+\eps
\Id$ and, by property \eqref{d1'}, $n_0(M(\eps)_\xi A_\xi) \equiv
n_0(M(\eps)_\xi A_\xi)=\sigma_0(A_\xi)=0$.
Mimicking the proof of Lemma \ref{Ever}, we find, similarly, that
the center subspace of $M_\xi A_\xi$ consists of the zero-subspace
of $M_\xi A_\xi$, which is likewise empty.
Thus, no zeros cross (or reach) the imaginary axis during the homotopy
and we obtain the result in the limit as $\eps\to 0$, for any
$\xi\ne 0$ sufficiently small that $A_\xi$ does not have a kernel.
\er

\br\label{recast}
The above clarifies somewhat the mechanism behind observed modulational
dichotomy results; namely, the origins of such dichotomies are a priori
knowledge of modulational instability (encoded, through considerations
as above, or as in Remark \ref{arndrmk}, in unconstrained instability
of $A_{\omega_*}$) plus the central property \eqref{e1} that the center
subspace of $MA_{\omega_*}$ consists entirely of $\ker MA_{\omega_*}$.
\er

\section{Discussion}\label{dis}

We conclude by briefly discussing how our results connect to other
stability studies.

\subsection{Relation to the Evans function}
First, we note the following connection to the Melnikov integral and the
Evans function:
\br If $r=1$, the quantity
$$
(\partial c/\partial \omega) =  \langle\nabla c,  (\partial X/\partial
\omega)\rangle= - \langle  A_\omega (\partial X/\partial \omega), (\partial
X/\partial \omega )\rangle
$$
is a Melnikov integral involving variations in $\omega$ along the
manifold of nearby stationary solutions, and as pointed out in
\cite{PW,Z2} and elsewhere should correspond to the first
nonvanishing derivative of an associated Evans function. \er

Of course, we don't need this connection, since we have an if and
only if condition for asymptotic stability under our assumption
\eqref{e1}, which is perhaps an analogy in the dissipative context
to the \cite{GSS} assumption in the Hamiltonian setting that $J$ be
one to one (the difference being that \eqref{e1} seems to be
satisfied for all the systems we know of/are interested in, while
$J$ one-to-one does not hold for (gKdV) and other primary examples).

\subsection{Localized structures: problems on the whole line}\label{s:spikes}

The results of \cite{Ho1,PS1,Z3} on exponential instability of solitary, or
``pulse-type'' exponentially localized spikes solutions suggest another dichotomy different
from the one we proved in Section~\ref{s:sideband}: if essential
spectrum of the linearization $L=-MA_{\omega^\ast}$ along the spike
is good, then point spectrum must be bad. The two dichotomies are
apparently different, but do resemble each other ---  for, co-periodic
stability is related to
point spectrum, sideband to essential spectrum
(thinking of the homoclinic limit \cite{G}).\footnote{
In view of the discussion of Section \ref{s:revisit},
the link may be just that modulation and the large-period limit are
related through small-$\xi$, or ``long-wave'' response, both
having the effect of ``removing'' constraints.
}

Moreover, we can recover
(slightly weakened versions of)
the above mentioned results on instability
of spikes from \cite{Ho1,PS1,Z3} using our results on variational
stability of periodic solutions.
Below we will illustrate the main ideas in the specific
case of the viscoelasticity model \eqref{e:visco} with $d=1$. First,
we assume that the essential spectrum of the linearization along the
spike is stable, otherwise the instability is trivial.  We note that
the quasi-gradient structure implies that it is enough to
prove variational instability  in order to obtain (neutral) time-evolutionary
instability. In \cite{Z3,Y} it was shown that for the viscoelasticity
with strain-gradient effects model \eqref{e:visco} the two types of
instability are actually equivalent.

We suspect this is true for a
larger class of models, but choose not to pursue this here.
Rather, we just notice that instability of $A_{\omega^*}$, {\it without
constraint} implies already an exponentially localized initial data
for the evolution problem on the whole line for which $\mathcal{E}$
is nondecreasing in time and initially strictly less that $\mathcal{E}(X^*)$,
whereas the manifold of nearby solitary wave solutions, consisting entirely
of translates of $X^*$, has energy $\equiv \mathcal{E}(X^*)$.
Thus, by continuity, $X(t)$ cannot converge in $H^2$ to the set of
translates of $X^*$, else the energies would also converge; this shows
that the standard notion (see \cite{Z3,Y}) of
 $L^1\cap H^s\to H^s$ orbital asymptotic stability cannot hold
for $s\ge 2$.

Next, we remark that any homoclinic spike solution can be approximated by a
sequence of periodic patterns having large period $2T$.  This is
true for any system that has a Hamiltonian structure, in particular
the models \eqref{s:CH}, \eqref{Cons-Diff} and \eqref{e:visco}.
Then, we use the convergence results for the point spectrum from
\cite{G,SS} to prove the (above type of neutral)
instability of the homoclinic spike
solution by showing that the approximating periodic pattern is
unstable for any $T> T_\ast>0$.

Averaging in \eqref{int-ve}, one readily checks that
\begin{equation}\label{conv1}
\omega_\tau=-\frac{1}{2T}\int_{-T}^{T}  W'(\tau^\ast(x))\, dx\to - W'(\tau^\infty),\quad\mbox{as}\quad T\to\infty.
\end{equation}
Similarly, using the definition of the constraint function, we have that
\begin{equation}\label{conv2}
\frac{c_\tau}{2T}=\frac{1}{2T}\int_{-T}^{T} \tau^\ast(x)\, dx\to \tau^\infty,\quad\mbox{as}\quad T\to\infty.
\end{equation}
From \eqref{conv1} and \eqref{conv2} we conclude that
\begin{equation}\label{conv3}
T\frac{\pa\omega_\tau}{\pa c_\tau}\to  -W''(\tau^\infty),\quad\mbox{as}\quad T\to\infty.\footnote{
As the reverese limit $dc/d\omega \to \infty$ indicates,
solutions with different multipliers $\omega$ are infinitely far apart
in the whole-line problem, and do not play a role; in particular,
$\ker MA=\Span \{X^*_x\}$ is one-dimensional.
}
\end{equation}
Since $\tau^\ast$ is a solution of \eqref{int-ve} having finite limit at $\pm\infty$ we infer  that $W''(\tau^\infty)>0$.
Our claim follows immediately from  Proposition~\ref{V-E-stability} (i) and \eqref{conv3},
yielding a sequence of eigenvalues $\nu_T>0$ converging as $T\to \infty$ to an
eigenvalue $\nu$ of $A_{\omega^*}$.
With further work, exploiting the results in \cite{SS}, one may show, provided the limiting operator
$A_{\omega^*}$ has a spectral gap at $\nu=0$, that the associated
eigenfunctions $f_T$ have uniform exponential decay at $\pm \infty $, hence
extract a subsequence converging to an eigenfunction $f$ of $A_{\omega^*}$; moreover,
the same uniform exponential decay allows us to conclude that
$f$ is both orthogonal to the translational eigenfunction $X^*_x$
and satisfies the linearized constraint, with eigenvalue $\nu\ge 0$
corresponding to ``neutral'' (i.e., nonstrict) orbital instability as claimed.
Provided that $\kernel(A_{\omega^*})$ is spanned by the
translational mode $X^*_x$, we obtain in fact {\it strict variational
instability} $\nu>0$.

On the other hand, we have, more simply, just by Sturm--Liouville
considerations, that $A$ is variationally unstable without constraints,
without the above construction.  Either argument leads to the conclusion
that asymptotic orbital stability cannot hold.  If one can verify further
that the derivative of the standard Evans function does not vanish at
$\lambda=0$, then one could go further (by property (E1))
to conclude strict, or exponential instability as shown in \cite{Z3}
(and analogously for reaction diffusion--conservation law spikes
in \cite{PS2}).
Such a result would then yield by \cite{G} the known result of co-periodic
instability for sufficiently large period, completing the circle of
arguments.

\br\label{frontrmk}
As noted in \cite{Z3}, stability of front-type solutions may also be
studied variationally for viscoelasticity with strain-gradient effects. It
would be interesting to try to phrase this entirely in terms of the
structure \eqref{gengrad}; in particular, we suspect that the
similar result of \cite{LW} in the context of chemotaxis is another
face of the same basic mechanism.
\er

\br\label{convert}
In the cases considered in \cite{Ho1,PS1,Z3}, the existence problem
is scalar second order with spectral gap at $\nu=0$, whence it is easily
by dimensionality seen that $X^*_x$ is the only element (up to constant
multiple) of $\kernel(A_{\omega^*})$, hence we indeed recover strict constrained
variational instability from the limiting periodic argument above.
However, we do not see a correspondingly easy way to see a priori that the
derivative of the Evans function does not vanish at the origin, so for
the moment obtain by this argument only {\it neutral instability}
for the time-evolutionary problem.
\er

\br\label{crit}
The above instability arguments (including
those of \cite{Ho1,PS1,Z3}) rely on spectral gap of
$A_{\omega^*}$ at the origin, $\nu=0$, and the associated property
of exponential convergence of $X^*$
to its endstate as $x \to \pm \infty$.
In the critical case of an algebraically-decaying homoclinic solution,
it has been shown for the Keller--Segel equation that
stability holds, with algebraic decay rate \cite{CF,BCC}.
\er

\subsection{Coupled conservative--reaction diffusion spikes
on bounded domains}\label{s:fspikes} As pointed out in \cite{PS2},
coupled conservation--reaction diffusion spikes have been
numerically observed \cite{Dr,HHMO,KS,KR}, suggesting that they may
sometimes be stable on finite domains. As we pointed out earlier,
the stability with periodic boundary conditions is equivalent to
stability with Neumann boundary conditions, since odd eigenfunctions
are automatically stable, due to the fact that the translational
derivative possesses only two sign changes in a domain of minimal
period. Assuming monotonicity of the period, one concludes that
patterns are stable whenever $dc/d\omega>0$. Now letting the period
tend to infinity, there are just a few possibilities for a periodic
solution: convergence to a homoclinic solution, convergence to a
heteroclinic solution, or unbounded amplitude. In fact, all
scenarios occur in relevant circumstances. We already discussed the
case of the spike limit, in which all periodic solutions are
unstable. In the Cahn-Hilliard example with, say, a cubic
nonlinearity $W'(u)=-u+u^3$, one finds stable solutions periodic
solutions that converge to a heteroclinic loop as the period tends
to infinity. In fact, the energy is coercive in this case, so that
there necessarily exists a constrained minimizer for any prescribed
mass. For masses $|C(u)|<1/\sqrt{3}$, homogeneous equilibria are
unstable so that coercivity of the energy enforces the existence of
stable equilibria for all periods, which, given a priori bounds from
the energy, therefore necessarily converge to a heteroclinic loop
(or a pair of layers/kinks). Similar considerations apply to the
coupled conservative reaction-diffusion context when we have a
globally coercive energy. In the case of chemotaxis, this global
coercivity fails. Also, the nonlinear oscillator that describes
periodic solutions possesses only two equilibria, which excludes
heteroclinic loops. In the large-wavelength limit, one finds
spike-like solutions which \emph{do not} converge to a homoclinic
solutions, but diverge along a family of spikes to infinity as the
period goes to infinity, consistent with our previous discussion.

\medskip
{\bf Acknowledgment.} Thanks to Jared Bronski for a helpful conversation
on origins/background literature to constrained variational problem
and related spectral theory,
and for informing us of his
unpublished work with Mathew Johnson on co-periodic stability of
nonlocal Cahn--Hilliard solutions.

\end{document}